\documentclass[11pt]{article}
\usepackage{amssymb,amsbsy,amsmath,amsfonts,amssymb,amscd,amsthm}
\usepackage{stmaryrd}
\usepackage{mathrsfs}
\usepackage[active]{srcltx}
\usepackage{bbm}

\usepackage{paralist}
\usepackage{graphics} 
\usepackage{epsfig} 
\usepackage{epstopdf} 
\usepackage[colorlinks=true]{hyperref}
\hypersetup{urlcolor=blue, citecolor=red}

\let\epsilon\varepsilon
\let\phi\varphi
\def\proof{\medskip\noindent{\bf Proof. }}

\def\into{\longrightarrow}
\def\Im{{\rm Im}}

\newfont\rsfseleven{rsfs10 scaled 1100}

\def\LL{\mbox{\normalfont\rsfseleven L}}
\def\BB{\mbox{\normalfont\rsfseleven B}}

\numberwithin{equation}{section}

\newtheorem{theorem}{Theorem}[section]
\newtheorem{lemma}[theorem]{Lemma}
\newtheorem{proposition}[theorem]{Proposition}
\newtheorem{corollary}[theorem]{Corollary}
\theoremstyle{definition}

\newtheorem{remark}[theorem]{Remark}

\begin{document}
\title{Polynomial Stabilization of Solutions to a\\ Class of Damped Wave Equations
\thanks{This work was supported by the National Natural Science Foundation of China (grant No. 60974033) and Beijing Municipal Natural Science Foundation (grant No. 4132051).
\medskip} }

\author{Otared Kavian\thanks{Laboratoire de Math\'ematiques Appliqu\'ees, UMR 8100 du CNRS, Universit\'e Paris--Saclay (site UVSQ),
45 avenue des Etats Unis, 78035 Versailles Cedex, France. Email: kavian@math.uvsq.fr}, and Qiong Zhang\thanks{School of
Mathematics, Beijing Institute of Technology, Beijing, 100081,
China (zhangqiong@bit.edu.cn).}}
\date{}
\maketitle

\begin{center}
\begin{minipage}{5.5in}
{\bf Abstract.}
We consider a class of wave equations of the type $\partial_{tt} u + Lu + B\partial_{t} u = 0$, with a self-adjoint operator $L$, and various types of local damping represented by $B$. By establishing appropriate and raher precise estimates on the resolvent of an associated  operator $A$ on the imaginary axis of ${{\Bbb C}}$, we prove  polynomial decay of the semigroup $\exp(-tA)$ generated by that operator. We point out that the rate of decay depends strongly on the concentration of eigenvalues and that of the eigenfunctions of the operator $L$.  We give several examples of application of our abstract result, showing in particular that for a rectangle $\Omega := (0,L_{1})\times (0,L_{2})$ the decay rate of the energy is different depending on whether the ratio $L_{1}^2/L_{2}^2$ is rational, or irrational but algebraic.
\bigskip

\centerline{\bf Version of March 2017}
\date

\bigskip


\bigskip

 \vskip 3mm

\end{minipage}
\end{center}


\section{Introduction}\label{sec-Intro}

\indent In this paper, we study the long time behavior of a class of wave equations with various types of damping (such as Kelvin-Voigt damping,   viscous damping, or both).
More precisely, let $N \geq 1$ be an integer, and let $\Omega\subset {\Bbb R}^N$ be a bounded Lipschitz domain, its boundary being denoted by $\partial\Omega$. The wave equations we study in this paper are of the following type
\begin{equation}\label{eq:System-1}
\begin{cases}
u_{tt} - \Delta u +
b_{1}(x) u_{t} -
{\rm div}\left(b_{2}(x)\nabla u_{t} \right) = 0  & \mbox{ in } (0, \infty) \times \Omega, \\
u(t,\sigma) = 0  & \mbox{ on } (0, \infty) \times \partial\Omega, \\
u(0,x)  = u_{0}(x) & \mbox{ in } \Omega \\
u_{t}(0,x)  = u_{1}(x) & \mbox{ in } \Omega.
\end{cases}
\end{equation}
In the above equation $\Delta$ is the Laplace operator on ${\Bbb R}^N$ and we denote $u_{t} := \partial_{t}u$ and $u_{tt} := \partial_{tt}u$, while $b_{1},b_{2} \in L^\infty(\Omega)$ are two nonnegative functions such that at least one of the following conditions 
\begin{equation}\label{eq:Cnd-b1}
\exists\,\epsilon_{1}>0 \;  \mbox{  and }  \;  \emptyset \neq \Omega_{1} \subset \Omega, \;  \mbox{ s.t.  }  \;  \Omega_{1}\; \mbox{ is open }\;  \mbox{  and }  \; b_{1} \geq \epsilon_{1}\; \mbox{on }\Omega_{1}
\end{equation}
or
\begin{equation}\label{eq:Cnd-b2}
\exists\,\epsilon_{2} > 0\;  \mbox{  and }  \; \emptyset \neq \Omega_{2} \subset \Omega, \;  \mbox{ s.t.  }  \;  \Omega_{2}\;\mbox{ is open }\;  \mbox{  and }  \; b_{2} \geq \epsilon_{2}\; \mbox{on }\Omega_{2}
\end{equation}
is satisfied. When $b_{1} \equiv 0$ and condition \eqref{eq:Cnd-b2} is satisfied, the wave equation described in \eqref{eq:System-1} corresponds to a wave equation with local viscoleastic damping on $\Omega_{1}$, that is a damping of Kelvin--Voigt's type (see, e.g., S.~Chen, K.~Liu \& Z.~Liu \cite{Chen-Liu-Liu}, K.~Liu \& Z.  Liu \cite{Liu-Liu}, M.~Renardy~\cite{Renardy}, and references therein). The case in which $b_{2} \equiv 0$, and \eqref{eq:Cnd-b1} is satisfied, corresponds to a damped wave equation where the damping, or friction, is activated on the subdomain $\Omega_{1}$, and is proportional to the velocity $u_{t}$ (see, e.g., C. Bardos, G.~Lebeau \& J.~Rauch \cite{Bardos-LR}, G.~Chen, S.A.~Fulling, F.J.~Narcowich \& S.~Sun \cite{Chen-Fulling}, and references therein).

The energy function associated to the system \eqref{eq:System-1} is
\begin{equation}\label{eq:energy}
E(t) =\frac{1}{2} \int_{\Omega}|u_t(t,x)|^2dx +
{1 \over 2 }\int_{\Omega}|\nabla u(t,x)|^2 \,dx.
\end{equation}
and it is dissipated according to the following relation:
\begin{equation}\label{eq:Energy-diss}
{{d}\over{dt}}E(t) = - \int_{\Omega}b_{1}(x)|u_{t}(t,x)|^2\, dx-\int_{\Omega} b_{2}(x) |\nabla u_t(t,x) |^2 dx.
\end{equation}
If $\Omega_{1} = \Omega$, that is a damping of viscous type exists on the whole domain, it is known that  the associated semigroup is exponentially stable (G.~Chen \& al \cite{Chen-Fulling}).
It is also known that the Kelvin-Voigt damping is stronger than the viscous damping, in the sense that if $\Omega_{2} = \Omega$, the damping for the wave equation not only induces exponential energy decay, but also restricts the spectrum of the associated semigroup generator to a sector in the left half plane, and the associated semigroup is analytic (see e.g.\ S.~Chen \& al \cite{Chen-Liu-Liu} and references therein).

When $b_{2} \equiv 0$ and the viscous damping is only present on a subdomain $\Omega_{1} \neq \Omega$, it is known that geometric optics conditions  guarantee the exact controllability, and
the exponential stability of the system (C.~Bardos \& al. \cite{Bardos-LR}).
However, when $b_{1} \equiv 0$, and $\Omega_{2} \neq \Omega$, even if $\Omega_{2}$ satisfies those geometric optics conditions, the  Kelvin-Voigt damping model does not necessarily have an exponential decay.  In fact, for the one dimensional case $N=1$, S.~Chen \& al. \cite{Chen-Liu-Liu} have proved that when $b_{2} := 1_{\Omega_{2}}$ (and $\Omega_{2} \neq \Omega$), the energy of the Kelvin--Voigt  system \eqref{eq:System-1} does not have an exponential decay.
A natural question is to study the decay properties of the wave equation with local viscoelastic damping, or when viscous damping is local and geometric optics conditions are not satisfied.
\medskip

Our aim is to show that, if one of the conditions \eqref{eq:Cnd-b1} or \eqref{eq:Cnd-b2} is satisfied then, as a matter of fact, the energy functional decreases to zero at least with a polynomial rate: more precisely, there exists a real number $ m > 0$ and a positive constant $c > 0$ depending only on $\Omega_{1},\Omega_{2},\Omega$ and on $b_{1},b_{2}$, such that
$$E(t) \leq {c \over (1 + t)^{2/m}}\, \left(\|\nabla u_{0}\|^2 + \|u_{1}\|^2 \right). $$
The positive number $m$ depends in an intricate way on the distribution of the eigenvalues $(\lambda_{k})_{k \geq 1}$ of the Laplacian with Dirichlet boundary conditions on $\partial\Omega$, and it depends also strongly on the concentration, or {\it localization\/}, properties of the corresponding eigenfunctions, and thus on the geometry of $\Omega$, and those of $\Omega_{1}, \Omega_{2}$. More precisely, let $(\lambda_{k})_{k \geq 1}$ be the sequence of eigenvalues given by
$$-\Delta \phi_{k,j} = \lambda_{k}\phi_{k,j}\quad\mbox{in }\,\Omega,\qquad \phi_{k,j}\in H^1_{0}(\Omega), \quad \int_{\Omega}\phi_{k,j}(x)\phi_{\ell,i}(x)dx = \delta_{k\ell}\delta_{ij}.$$
Here we make the convention that each eigenvalue has multiplicity $m_{k} \geq 1$ and that in the above relation $1 \leq j \leq m_{k}$ and $1\leq i \leq m_{\ell}$. As usual, the eigenvalues $\lambda_{k}$ are ordered in an increasing order: $0 < \lambda_{1} < \lambda_{2} < \cdots < \lambda_{k} < \lambda_{k+1}<\cdots$. We shall consider the cases where an exponent denoted by $\gamma_{1} \geq 0$ exists such that for some constant $c_{0} > 0$ one has
\begin{equation}\label{eq:Cnd-Gap-1}
\forall\, k \geq 2,\qquad \min\left({\lambda_{k} \over \lambda_{k-1}} -1,
1 - {\lambda_{k} \over \lambda_{k+1}}\right) \geq c_{0} \lambda_{k}^{-\gamma_{1}}.
\end{equation}
We shall need also the following assumption on the concentration properties of the eigenfunctions $\phi_{k,j}$: there exist a constant $c_{1} > 0$ and an exponent $\gamma_{0} \in {\Bbb R}$ such that if $\phi := \sum_{j=1}^{m_{k}}c_{j}\phi_{k,j}$ with $\sum_{j=1}^{m_{k}}|c_{j}|^2 = 1$, then
\begin{equation}\label{eq:Concentration-1}
\forall\, k \geq 2,\qquad \int_{\Omega_{1}}|\phi(x)|^2\,dx +
\int_{\Omega_{2}}|\nabla\phi(x)|^2\,dx \geq c_{1}\, \lambda_{k}^{-\gamma_{0}}.
\end{equation}
Then we show that if $m := 3 + 2\gamma_{0} + 4 \gamma_{1}$, the decay of the energy is at least of order $(1 + t)^{-2/m}$. The fact that the above conditions are satisfied for certain domains $\Omega$, and subdomains $\Omega_{1}, \Omega_{2}$, depends strongly on the geometry of these domains and will be investigated later in this paper, by giving examples in which these assumptions are satisfied.
\medskip

Before stating our first result, which will be in an abstract setting and will be applied to several examples later in this paper, let us introduce the following notations. We consider an infinite dimensional, complex, separable Hilbert space $H_{0}$, and a positive (in the sense of forms) self-adjoint operator $(L,D(L))$ acting on $H_{0}$, which has a compact resolvent (in particular $D(L)$ is dense in $H_{0}$ and $(L,D(L))$ is an unbounded operator). The dual of $H_{0}$ having been identified with $H_{0}$, we define the spaces $H_{1}$ and $H_{-1}$ by
\begin{equation}\label{eq:Def-H-1}
H_{1} := D(L^{1/2}) \qquad\mbox{and}\qquad H_{-1} := (H_{1})',
\end{equation}
the space $H_{1}$ being endowed with the norm $u\mapsto \|L^{1/2}u\|$. 
We shall denote also by $\langle \cdot,\cdot\rangle$ the duality between $H_{-1}$ and $H_{1}$, adopting the convention that if $f \in H_{0}$ and $\phi \in H_{1}$, then
$$\langle f, \phi \rangle = (\phi|f).$$
As we mentioned earlier, we adopt the convention that the spectrum of $L$ consists in a sequence of distinct eigenvalues $(\lambda_{k})_{k \geq 1}$, with the least eigenvalue $\lambda_{1} > 0$, numbered in an increasing order and $\lambda_{k} \to +\infty$ as $k\to\infty$, each eigenvalue $\lambda_{k}$ having multiplicity $m_{k} \geq 1$.

Next we consider an operator $B$ satisfying the following conditions:
\begin{equation}\label{eq:Cnd-Op-B-1}
\begin{cases}
B : H_{1} \into  H_{-1} \,\mbox{ is bounded,} \\
B^* = B,\; \mbox{i.e.} \quad \langle B\phi,\psi\rangle = \overline{\langle B\psi,\phi\rangle} \; \mbox{ for }\,\phi,\psi\in H_{1},\\
\forall\, \phi\in H_{1}, \qquad \langle B\phi, \phi \rangle \geq 0.
\end{cases}
\end{equation}
We will assume moreover that $B$ satisfies the non degeneracy condition
\begin{equation}\label{eq:Cnd-Op-B-2}
\forall k \geq 1, \quad
{1 \over \beta_{k}} := \min\left\{ \langle B\phi,\phi\rangle \; ; \;
\phi \in N(L - \lambda_{k}I),\; \|\phi\| = 1 \right\} > 0,
\end{equation}
and we study the abstract second order equation of the form
\begin{equation}\label{eq:System-Abs}
\begin{cases}
u_{tt} + L u + Bu_{t} = 0  & \mbox{ in } (0, \infty)  \\
u(t) \in D(L) &\\
u(0)  = u_{0} \in D(L) &\\
u_{t}(0) = u_{1} \in H_{1} .&
\end{cases}
\end{equation}
Our first result is:

\begin{theorem}\label{lem:Main-1} Assume that the operators $(L,D(L))$ and $B$ are as above, and let the eigenvalues $(\lambda_{k})_{k \geq 1}$ of $L$ satisfy 
\begin{equation}\label{eq:Poly-growth-eigen}
\lambda_{*} := \inf_{k \geq 1} {\lambda_{k} \over \lambda_{k+1}} > 0.
\end{equation}
Moreover assume that there exist a constant $c_{0} >0$ and two exponents $\gamma_{0} \in {\Bbb R}$ and $\gamma_{1} \geq 0$ such that, $\beta_{k}$ being defined by \eqref{eq:Cnd-Op-B-2}, for all inetgers $k \geq 1$ we have
\begin{eqnarray}
&\beta_{k} \leq c_{0} \, \lambda_{k}^{\gamma_{0}}\, , \label{eq:Cnd-beta-k} \\
& \displaystyle {\lambda_{k-1} \over \lambda_{k} - \lambda_{k-1}} + {\lambda_{k+1} \over \lambda_{k+1} - \lambda_{k}} \leq c_{0}\, \lambda_{k}^{\gamma_{1}}. \label{eq:Cnd-alpha-k}
\end{eqnarray}
Then, setting $m := 3 + 2\gamma_{0} + 4\gamma_{1}$, there exists a constant $c_{*} > 0$ such that for all $(u_{0},u_{1})\in D(L)\times H_{1}$, and for all $t > 0$, the energy of the solution to equation \eqref{eq:System-Abs} satisfies
\begin{equation}\label{eq:Decay-Abs}
(Lu(t)|u(t)) + \|u_{t}(t)\|^2 \leq c_{*} (1+t)^{-2/m}\left[ (Lu_{0}|u_{0}) + \|u_{1}\|^2\right]\, .
\end{equation}
\end{theorem}

Our approach consists in establishing, by quite elementary arguments, rather precise a priori estimates on the resolvent of the operator
\begin{equation}
u := (u_{0},u_{1}) \mapsto Au := \left(-u_{1}, L u_{0} + B u_{1} \right)
\end{equation}
on the imaginary axis ${\rm i}\,{\Bbb R}$ of the complex plane. Indeed the proof of Theorem \ref{lem:Main-1} is based on results characterizing the decay of the semigroup $\exp(-tA)$ in terms of bounds on the norm of the resolvent $\|(A - {\rm i}\,\omega)^{-1}\|$ as $|\omega| \to \infty$ (see J.~Pr\"uss \cite{PrussJ-1984}, W.~Arendt \& C.J.K.~Batty \cite{ArendtW-Batty-1988}, Z.~Liu \& B.P.~Rao \cite{Liu-Rao-2005},
 A.~B\'atkai, K.J.~Engel, J.~Pr\"uss and R.~Schnaubelt \cite{Batkai-Engel-Pruss-2006}, C.J.K.~Batty \& Th.~Duyckaerts \cite{BattyCJK-Duyckaerts-2008}). In this paper we use the following version of these results due to A.~Borichev \& Y.~Tomilov \cite{BorichevA-Tomilov-2010}:

\begin{theorem} \label{lem:Borichev-Tomilov}
Let $S(t) := \exp(-tA)$ be a $C_{0}$-semigroup on a Hilbert space ${\Bbb H}$ generated by the operator $A$. Assume that ${\rm i}\,{\Bbb R}$ is contained in $\rho(A)$, the resolvent set of $A$, and denote $R(\lambda,A) := (A - \lambda I)^{-1}$ for $\lambda \in \rho(A)$. Then for $m > 0$ fixed, one has:
\begin{equation}
\sup_{\omega \in {\Bbb R}}{\|R({\rm i}\,\omega,A)\| \over (1 + |\omega|)^{1/m}} < \infty \iff
\sup_{ t\geq 0} \, (1 + t)^{m}\|S(t)A^{-1}\| < \infty.
\end{equation}
\end{theorem}
\medskip

The remainder of this paper is organized as follows. In Section \ref{sec:Abs} we gather a few notations and preliminary results concerning equation \eqref{eq:System-1}, and we prove Theorem \ref{lem:Main-1} by establishing a priori estimates for the resolvent of the generator of the semigroup associated to \eqref{eq:System-1}, in order to use the above Theorem \ref{lem:Borichev-Tomilov}. In Section \ref{sec:Wave-dim-1} we apply our abstract result to various wave equations in dimension one, and in Section \ref{sec:Wave-dim-2} we give a few examples of damped wave equations in higher dimensions when $\Omega := (0,L_{1})\times\cdots(0,L_{N})$, with $N \geq 2$, which show that depending on algebraic properties of the numbers $L_{i}^2/L_{j}^2$ the decay rate of the energy may be different.

\section{An abstract result}\label{sec:Abs}

As stated in the introduction, in what follows $H_{0}$ is an infinite dimensional, separable, complex Hilbert space, whose scalar product and norm are denoted by $(\cdot|\cdot)$ and $\|\cdot\|$, on which we consider a positive, densely defined selfadjoint operator $(L(D(L)))$ acting on $H_{0}$.
With the definition \eqref{eq:Def-H-1} we have $H_{1}\subset H_{0} = (H_{0})' \subset H_{-1}$, with dense and compact embeddings, and $L$ can be considered as a selfadjoint isomorphism (even an isometry) between the Hilbert spaces $H_{1}$ and $H_{-1}$.
We denote ${\Bbb R}^* := {\Bbb R}\setminus \{0\},\; {\Bbb N}^* := {\Bbb N}\setminus \{0\}$.

By an abuse of notations, ${\Bbb X}$ being a Banach space, when there is no risk of confusion we may write $f := (f_1, f_2) \in {\Bbb X}$ to mean that both functions $f_{1}$ and $f_{2}$ belong to ${\Bbb X}$ (rather than writing $f \in {\Bbb X} \times {\Bbb X}$ or $f \in {\Bbb X}^2$). Thus we shall write also $(f|g) := (f_{1}|g_{1}) + (f_{2}|g_{2})$ for $f = (f_{1},f_{2}) \in H_{0} \times H_{0}$ and
$g = (g_{1},g_{2}) \in H_{0} \times H_{0}$. Analogously $\|f\|$ will stand for $\left(\|f_{1}\|^2 + \|f_{2}\|^2\right)^{1/2}$.
\medskip

Recall that we have denoted by $(\lambda_{k})_{k \geq 1}$ the increasing sequence of distinct eigenvalues of $L$, each eigenvalue $\lambda_{k}$ having multiplicity $m_{k} \geq 1$. We shall denote by ${\Bbb P}_{k}$ the orthogonal projection of $H_{0}$ on the eigenspace $N(L - \lambda_{k}I)$. We denote by
\begin{equation}\label{eq:Def-F-k}
{\Bbb F}_{k} := \bigoplus_{j \geq k + 1} N(L - \lambda_{j} I) \, ,
\end{equation}
and
\begin{equation}\label{eq:Def-E-k}
{\Bbb E}_{k} := \left(N(L - \lambda_{k}I) \oplus {\Bbb F}_{k}\right)^\perp = \bigoplus_{1 \leq j \leq k - 1} N(L - \lambda_{j} I) \, ,
\end{equation}
(note that ${\Bbb E}_{1} = \left\{0\right\}$). We recall that ${\Bbb E}_{k}$ and ${\Bbb F}_{k}$ are invariant under the action of $L$, and obviously under that of $L - \lambda_{k}I$. More precisely
$$L({\Bbb E}_{k}) = (L - \lambda_{k}I)({\Bbb E}_{k}) = {\Bbb E}_{k}, $$
and also
$$L\left(D(L) \cap {\Bbb F}_{k}\right) =
(L - \lambda_{k}I)\left(D(L) \cap {\Bbb F}_{k}\right) = {\Bbb F}_{k}.$$

Next we consider a bounded linear operator $B$ satisfying the conditions \eqref{eq:Cnd-Op-B-1},
and we recall that one has a Cauchy-Schwarz type inequality for $B$, more precisely
\begin{equation}\label{eq:CS-ineq}
\left|\langle B\phi,\psi \rangle \right| \leq  \langle B\phi,\phi \rangle^{1/2}\,
 \langle B\psi,\psi \rangle ^{1/2}, \qquad\forall\, \phi, \psi\,\in H_{1}.
\end{equation}
which implies in particular that if $\langle B\phi,\phi \rangle = 0$ then we have $B\phi = 0.$

We shall consider the abstract damped wave equation of the form \eqref{eq:System-Abs}, and we introduce the Hilbert space
\begin{equation}
{\Bbb H} := H_{1} \times H_{0},
\end{equation}
corresponding to the energy space associated to equation \eqref{eq:System-Abs}, whose elements will be denoted by $u := (u_{0},u_{1})$ and whose norm is given by
$$
 \|u\|_{{\Bbb H}}^2  =  \|L^{1/2} u_{0}\|^2 + \|u_{1}\|^2.
$$
In order to solve and study \eqref{eq:System-Abs}, we define an unbounded operator $(A, D(A))$ acting on ${\Bbb H}$ by setting
\begin{equation}\label{eq:Def-A-1}
Au := (-u_{1}, L u_{0} + Bu_{1}),
\end{equation}
for $u \in D(A)$ defined to be
\begin{equation}\label{eq:Def-Dom-A-1}
D(A) :=
\left\{u \in {\Bbb H} \; ; \; u_{1} \in H_{1},\; L u_{0} + Bu_{1} \in H_{0} \right\}.
\end{equation}
Since for $u =(u_{0},u_{1}) \in D(A)$ we have
$$(Au|u)_{{\Bbb H}} = \langle Bu_{1},u_{1} \rangle \geq 0,$$
one can easily see that the operator $A$ is $m$-accretive on ${\Bbb H}$, that is for any $\lambda > 0$ and any $f \in {\Bbb H}$ there exists a unique $u \in D(A)$ such that
$$\lambda Au + u = f,\qquad\mbox{and}\quad \|u\| \leq \|f\|.$$
Thus $D(A)$ is dense in ${\Bbb H}$, and $A$ is a closed operator generating  a $C_0$-semigroup acting on ${\Bbb H}$, denoted by $S(t) := \exp(-tA)$ (see for instance K.~Yosida \cite{YosidaK-book}, chapter IX).

Then, writing the system \eqref{eq:System-Abs} as a Cauchy problem in ${\Bbb H}$:
$${d U \over dt} + AU(t) = 0 \quad \mbox{for }\, t > 0, \quad
U(0) = U^{0} := (u_{0},u_{1}) \in {\Bbb H},$$
we have $U(t) =(U_0(t),\, U_1(t))= \exp(-tA)U^{0}$, and the solution of \eqref{eq:System-Abs} is given by $u(t) = U_{0}(t)$, the first component of $U(t)$.

In order to study the behavior of $u(t)$, or rather that of $U(t)$, as $t \to +\infty$, we are going to show that the resolvent set of $A$ contains the imaginary axis ${\rm i}\,{\Bbb R}$ of the complex plane and that, under appropriate assumptions on the operators $L$ and $B$, the norm  $\|(A - {\rm i}\,\omega I)^{-1}\|$ has a polynomial growth as $|\omega| \to \infty$.

\begin{lemma}
The adjoint of $(A,D(A))$ is given by the operator $(A^*,D(A^*))$ where
\begin{equation}\label{eq:Dom-Adj-A}
D(A^*) = \left\{v \in {\Bbb H} \; : \; v_{1} \in H_{1},\;
-L v_{0} + Bv_{1} \in H_{0} \right\},
\end{equation}
and for $v = (v_{0},v_{1}) \in D(A^*)$ we have
\begin{equation}\label{eq:Def-Adj-A}
A^*v = \left(v_{1}, -L v_{0} + Bv_{1} \right).
\end{equation}
\end{lemma}

\proof Since $D(A)$ is dense in ${\Bbb H}$, the adjoint of $A$ can be defined. Recall that $v =(v_{0},v_{1})\in D(A^*)$ means that there exists a constant $c > 0$ (depending on $v$) such that
$$\forall\, u =(u_{0},u_{1})\in D(A), \qquad |(Au|v)_{{\Bbb H}}| \leq c\, \|u\|_{{\Bbb H}}.$$
To determine the domain of $A^*$, let $v \in D(A^*)$ be given, and consider first an element $u = (u_{0},0) \in D(A)$. Thus $Au = (0,L u_{0}) \in {\Bbb H}$ and
$$|(Au|v)_{{\Bbb H}}| = |(L u_{0}|v_{1})| \leq c\, \|u\|_{{\Bbb H}} = c\, \|L^{1/2} u_{0}\|.$$
This means that the linear form $u_{0} \mapsto (L u_{0}|v_{1})$ extends to a continuous linear form on $H_{1}$, and this is equivalent to say that $v_{1} \in H_{1}$, and that for $u = (u_{0},0)\in {\Bbb H}$ we have
$$(Au|v)_{{\Bbb H}} = \langle L u_{0}, v_{1} \rangle = \overline{(v_{1}|u_{0})_{H_{1}}}\,.$$

Now take $u = (0,u_{1}) \in D(A)$. Since, according to \eqref{eq:Cnd-Op-B-1}, $B^* = B$, we have
$$(Au|v)_{{\Bbb H}} = (- u_{1}|v_{0})_{H_{1}} + \langle Bu_{1}, v_{1}\rangle =
- (L^{1/2}u_{1}|L^{1/2}v_{0}) + \overline{\langle Bv_{1},u_{1} \rangle},$$
and thus
$$(Au|v)_{{\Bbb H}} =  \overline{\langle -L v_{0} + Bv_{1},u_{1} \rangle}$$
and since $v\in D(A^*)$ means that the mapping $u_{1} \mapsto (Au|v)$ extends to a continuous linear form on $H_{0}$, we conclude that
$$-L v_{0} + Bv_{1} \in H_{0},$$
and
$$(Au|v)_{{\Bbb H}} =  \overline{(-L v_{0} + Bv_{1} | u_{1})} = (u_{1}|-L v_{0} + Bv_{1}).$$
From these observations it is easy to conclude that in fact
$$A^\ast v = (v_{1}, -L v_{0} + Bv_{1}),$$
and that the domain of $A^*$ is precisely given by \eqref{eq:Dom-Adj-A}.
\qed
\medskip

We shall use the following classical results of S.~Banach which characterizes operators having a closed range (see for instance K.Â Yosida \cite{YosidaK-book}, chapter VII, section~5):
\begin{theorem}\label{lem:Banach-closed}
Let $(A,D(A))$ be a densely defined operator acting on a Hilbert space ${\Bbb H}$. Then
$$R(A)\,\mbox{ closed } \iff R(A^*)\, \mbox{ closed },$$
and either of the above properties is equivalent to either of the following equivalent equalities
$$R(A) = N(A^*)^\perp \iff R(A^*) = N(A)^\perp.$$
Moreover when $N(A) = N(A^*) = \{0\}$, the range $R(A)$ is closed if and only if there exist two constants $c_{1}, c_{2} > 0$ such that
\begin{equation}\label{eq:Inverse-A}
\forall\, u \in D(A)\quad \|u\| \leq c_{1}\, \|Au\|, \qquad
\forall\, v \in D(A^*)\quad \|v\| \leq c_{2}\, \|A^*v\|.
\end{equation}
\end{theorem}
\medskip

In the following lemma we show that $0 \in \rho(A)$, that is that the operator $A$ defined by \eqref{eq:Def-A-1} has a bounded inverse:
\begin{lemma}\label{lem:A-invertible}
Denote by $N(A)$ the kernel of the operator $A$ defined by \eqref{eq:Def-A-1}--\eqref{eq:Def-Dom-A-1}, and by $R(A)$ its range. Then we have $N(A) = \{0\} = N(A^*)$ and $R(A)$, as well as $R(A^*)$, are closed.
In particular $A : D(A) \longrightarrow {\Bbb H}$ is one-to-one and its inverse is continuous on ${\Bbb H}$.
\end{lemma}

\proof It is clear that $N(A) = N(A^*) = \{0\}$. On the other hand, thanks to Banach's theorem \ref{lem:Banach-closed}, we have only to show that $R(A)$ is closed. For a sequence $u_{n} = (u_{0n},\, u_{1n}) \in D(A)$ such that $f_{n} = (f_{0n},\, f_{1n}):= Au_{n} \to f =  (f_{0},\, f_{1}) \in {\Bbb H}$, we have to show that there exists $u = (u_{0},\, u_{1})\in D(A)$ for which $f = Au$. Since $u_{1n} = - f_{0n} \to -f_{0}$ in $H_{1}$, setting $u_{1} := -f_{0}$, due to the fact that $B : H_{1} \into H_{-1}$ is continuous, we have that
$$L u_{0n} = f_{1n} - Bu_{1n} \to f_{1} - Bu_{1}\quad\mbox{in }\, H_{-1},  $$
and therefore, denoting by $u_{0} \in H_{1}$ the unique solution of
$$L u_{0} = f_{1} - Bu_{1},$$
we have that $u = (u_{0},u_{1}) \in D(A)$ and that $Au = f$. This proves that the range of $A$, as well as that of $A^*$, are closed. Thus we have $R(A) = N(A^*)^\perp$ and $R(A^*) = N(A)^\perp$. Since $N(A) = N(A^*) = \{0\}$, by property \eqref{eq:Inverse-A} of Theorem \ref{lem:Banach-closed}  there exist two constants $c_{1}, c_{2} > 0$ such that
$$\forall\, u \in D(A),\quad \|Au\| \leq c_{1}\, \|u\|, \qquad
\forall\, v \in D(A^*),\quad \|A^*v\| \leq c_{2}\, \|v\|.
$$
This means that $A^{-1} : {\Bbb H} \longrightarrow D(A)$ is continuous, and naturally the same is true of $(A^*)^{-1} : {\Bbb H} \longrightarrow D(A^*)$.
\qed

\medskip
Next we show that a certain perturbation of $L$, which appears in the study of the resolvent of $A$, is invertible.

\begin{proposition}\label{lem:Estim-1} {\bf (Main estimates)}.
Assume that $B$ satisfies \eqref{eq:Cnd-Op-B-1}, and that for any fixed $k \geq 1$ condition \eqref{eq:Cnd-Op-B-2} is satisfied. Let $\omega \in {\Bbb R}$, and for $j\geq 1$ and $\omega^2 \neq \lambda_{j}$ denote
\begin{equation}\label{eq:Def-alpha-k}
\alpha_{j}(\omega) := {\lambda_{j} \over |\omega^2 - \lambda_{j}|}.
\end{equation}
Then the operator $L_{\omega} : H_{1} \into H_{-1}$ defined by
\begin{equation}\label{eq:L-omega}
L_{\omega}u := Lu - {\rm i}\, \omega Bu -\omega^2 u.
\end{equation} 
has a bounded inverse, and
$$\|L_{\omega}^{-1}\|_{H_{-1} \to H_{1}} \leq c(\omega)$$
where  the constant $c(\omega)$ is given by 
\begin{equation}\label{eq:Def-c-omega}
c(\omega) := \left({\beta_{k}\lambda_{k} \over |\omega|} + (1 +\beta_{k}\lambda_{k}) (\alpha_{k-1}(\omega) + \alpha_{k+1}(\omega))^2(1 + |\omega|)\right)c^*\,,
\end{equation}
for $\omega$ such that $\lambda_{k-1} < \omega^2 < \lambda_{k+1}$, with $c^* := 16 (1+\|B\|)^2$, and $\|B\| := \|B\|_{H_{1}\to H_{-1}}$.
\end{proposition}

\proof Note that according to \eqref{eq:Def-alpha-k}, the constants $\alpha_{k-1}(\omega)$ and $\alpha_{k+1}(\omega)$ are well-defined whenever $\lambda_{k-1} < \omega < \lambda_{k+1}$.

For $\omega \in {\Bbb R}$ fixed, and any given $g\in H_{-1}$ we have to show that there exists a unique $u_{0}\in H_{1}$  solution of
\begin{equation}\label{eq:u0}
L u_{0} - {\rm i}\, \omega\, Bu_{0} -\omega^2u_{0} = g\, ,
\end{equation}
and there exists a constant $c(\omega) >0$ such that
\begin{equation}
\label{eq:Estim-L-omega}
\|L^{1/2} u_{0}\| \leq c(\omega)\, \|g\|_{H_{-1}}\, .
\end{equation}
Note that $L_{\omega}$ is a bounded operator from $H_{1}$ into $H_{-1}$ and that $(L_{\omega})^* = L_{-\omega} $. First we show that $N(L_{\omega}) = \{0\}$. If $\omega = 0$, then we know that $L_{0} = L$ and by assumption $N(L)=\{0\}$. If $\omega \neq 0,$ and if $u \in H_{1}$ satisfies $L_{\omega} u = 0$, we have
$$-\omega \langle Bu, u \rangle = \Im\,\langle Lu - {\rm i}\,\omega Bu - \omega^2 u, u\rangle = 0.$$
Since $\omega \neq 0$, this yields $\langle Bu, u \rangle = 0$ and, as remarked above after the Cauchy--Schwarz inequality \eqref{eq:CS-ineq}, the latter implies that $Bu = 0$ and thus $Lu - \omega^2 u = 0$. If $u$ were not equal to zero, this would imply that $u \in D(L)$ and that $\omega^2$ is an eigenvalue of $L$, say $\omega^2 = \lambda_{k}$ for some integer $k \geq 1$, that is $u \in N(L- \lambda_{k}I) \setminus \{0\}$. However we have $\langle Bu, u \rangle = 0$ and this in contradiction with the assumption \eqref{eq:Cnd-Op-B-2}. Therefore $u = 0$ and $N(L_{\omega}) = \{0\}$.
\medskip

Next we show that $R(L_{\omega})$ is closed, that is, according to property \eqref{eq:Inverse-A} of Banach's theorem \ref{lem:Banach-closed}, there exists a constant $c(\omega) > 0$ such that \eqref{eq:Estim-L-omega} is satisfied.

To this end, $\omega \in {\Bbb R}$ being fixed, we define two bounded operators $\BB$ and $\LL_{\omega}$ acting in $H_{0}$ by
\begin{eqnarray}
&\BB &:= L^{-1/2}BL^{-1/2} \label{eq:Def-BB} \\
&\LL_{\omega} &:= I - \omega^2 L^{-1} - {\rm i}\, \omega \BB = 
I - {\rm i}\,\omega L^{-1/2}(B - {\rm i}\,\omega I)L^{-1/2} , \label{eq:Def-LL}
\end{eqnarray}
and we note that 
$$L_{\omega} = L^{1/2}\left(I - {\rm i}\,\omega L^{-1/2}(B - {\rm i}\,\omega I)L^{-1/2}\right)L^{1/2} = L^{1/2}\LL_{\omega}L^{1/2}.$$
Since $L^{1/2}$ is an isometry between $H_{1}$ and $H_{0}$, and also between $H_{0}$ and $H_{-1}$, in order to see that $L_{\omega}^{-1}$ is a bounded operator mapping $H_{-1}$ into $H_{1}$, with a norm estimated by a certain constant $c(\omega)$, it is sufficient to show that the operator $\LL_{\omega}$, as a mapping on $H_{0}$, has an inverse and that 
$c(\omega)$ being defined in \eqref{eq:Def-c-omega} we have
\begin{equation}\label{eq:Estim-LL-omega}
\|\LL_{\omega}^{-1}\| \leq c(\omega).
\end{equation}

We observe also that $\BB : H_{0} \into H_{0}$ is a bounded, selfadjoint and nonnegative operator and thus, as recalled in \eqref{eq:CS-ineq}, for any $f,g \in H_{0}$ we have the Cauchy-Schwarz inequality
\begin{equation}\label{eq:CS-BB}
|(\BB f|g)| \leq (\BB f|f)^{1/2} (\BB g|g)^{1/2}.
\end{equation}
Now consider $f,g \in H_{0}$ such that $\|g\| \leq 1$ and
\begin{equation}\label{eq:LL-f}
\LL_{\omega}f = f - \omega^2L^{-1}f -{\rm i}\,\omega \BB f = g.
\end{equation}
We split the proof into two steps, according to whether $\omega^2$ is smaller or larger than $\lambda_{1}/2$.
\medskip

\noindent{\bf Step 1.} Assume first that $\omega^2 \leq \lambda_{1}/2$. Using the fact that
$$(L^{-1}f|f) \leq {1 \over \lambda_{1}}\|f\|^2,$$
upon multiplying \eqref{eq:LL-f} by $f$, and then taking the real part of the resulting equality, one sees that
\begin{equation}\label{eq:Estim-u0-1}
\|f\| \leq {\lambda_{1} \over \lambda_{1} - \omega^2}.
\end{equation}
Thus if $\omega^2 < \lambda_{1}$ one has $\|\LL_{\omega}^{-1}\| \leq \lambda_{1}/(\lambda_{1} -\omega^2)$, and more precisely $\|\LL_{\omega}^{-1}\| \leq 2$ if $\omega^2 \leq \lambda_{1}/2$.
\medskip

\noindent{\bf Step 2.} Now assume that for some integer $k \geq 1$ we have
\begin{equation}\label{eq:omega2-k}
\lambda_{k-1} <  \omega^2 < \lambda_{k+1} .
\end{equation}
(If $k=1$ by convention we set $\lambda_{0} :=0$). Multiplying, in the sense of $H_{0}$, equation \eqref{eq:LL-f} by $f$ and taking the imaginary part of the result yields
\begin{equation}\label{eq:BB-f}
(\BB f|f) \leq |\omega|^{-1}\, \|f\|.
\end{equation}
This first estimate is indeed not sufficient to obtain a bound on $f$, since the operator $\BB$ may be neither strictly nor uniformly coercive. However, as we shall see in a moment, this crude estimate is a crucial ingredient to obtain our result. 
\medskip

We begin by decomposing $f$ into three parts as follows: there exist a unique $t \in {\Bbb C}$ and  $\phi \in N(L - \lambda_{k}I)$, with $\|\phi\| = 1$, such that
$$f = v + t \phi +  z,\quad\mbox{where }\, v \in {\Bbb E}_{k},\,
\mbox{ and }\, z \in {\Bbb F}_{k}\cap H_{1}.$$
(Recall that we have ${\Bbb E}_{1} = \{0\}$; also when $\lambda_{k}$ has multiplicity $m_{k} \geq 2$, then $\phi$ may depend also on $g$, but in any case its norm in  $H_{1}$ is $\sqrt{\lambda_{k}}$). With these notations, equation \eqref{eq:LL-f} reads
\begin{equation}\label{eq:f-Decomp}
v - \omega^{2}L^{-1}v + z - \omega^{2}L^{-1}z + {\lambda_{k} - \omega^2 \over \lambda_{k}} \,t\phi = {\rm i}\,\omega\,\BB f + g.
\end{equation}
When $k=1$, by convention we have $v=0$, while when $k \geq 2$ we may multiply the above equation by $-v$, and using the fact that for $v\in {\Bbb E}_{k}$ we have 
$$(L^{-1}v|v) \geq {1 \over \lambda_{k-1}}\, \|v\|^2,$$ 
we deduce that
$${\omega^2 - \lambda_{k-1} \over \lambda_{k-1}}\, \|v\|^2 \leq \|v\| + |\omega| \, |(\BB f| v)|.$$
Using \eqref{eq:CS-BB} and \eqref{eq:BB-f} we have 
$$|(\BB f| v)| \leq (\BB f|f)^{1/2} (\BB v|v)^{1/2} \leq |\omega|^{-1/2}\|f\|^{1/2}\cdot \|\BB\|^{1/2}\cdot \|v\|,$$
so that since $\|\BB\|\leq \|B\| := \|B\|_{H_{1} \to H_{-1}}$, we get finally
\begin{equation}\label{eq:Estim-v}
\|v\| \leq \alpha_{k-1}(\omega)\left(1 + |\omega|^{1/2}\|B\|^{1/2}\|f\|^{1/2}\right).
\end{equation}
Analogously, multiplying \eqref{eq:f-Decomp} by $z$ and using the fact that 
$$(L^{-1}z|z)\leq {1 \over \lambda_{k+1}}\, \|z\|^2,$$ 
we get
$${\lambda_{k+1} - \omega^2 \over \lambda_{k+1}}\, \|z\|^2 \leq \|z\| + |\omega| \, |(\BB f| z)|,$$
and, proceeding as above, we deduce that
\begin{equation}\label{eq:Estim-z}
\|z\| \leq \alpha_{k+1}(\omega)\left(1 + |\omega|^{1/2}\|B\|^{1/2}\|f\|^{1/2}\right).
\end{equation}
Writing \eqref{eq:f-Decomp} in the form
$$(I - \omega^{2}L^{-1})(v + z) + {\lambda_{k} - \omega^2 \over \lambda_{k}}\,t\,\phi -{\rm i}\,t\,\omega\,\BB \phi = g + {\rm i}\,\omega\,\BB (v+z) ,$$
we multiply this equation by $\phi$ and we take the imaginary part of the resulting equality to obtain 
\begin{eqnarray} 
&|t|\, |\omega| (\BB \phi|\phi) &\leq 1 +  |\omega|\cdot|(\BB \phi|v+z)| \nonumber\\
&&\leq
\left(1 + |\omega|(\BB\phi|\phi)^{1/2} \|\BB\|^{1/2}(\|v\| + \|z\|)\right).
\label{eq:Estim-t2-B}
\end{eqnarray}
(Here we have used the fact that $((I - \omega^{2}L^{-1})(v + z)|\phi) = 0$ since $v+z \in N(L - \lambda_{k} I)^\perp$).  Now we have 
$$(\BB \phi|\phi) = \langle BL^{-1/2}\phi,L^{-1/2}\phi\rangle = {1 \over \lambda_{k}}\langle B \phi, \phi\rangle \geq {1 \over \beta_{k}\lambda_{k}}\, ,$$
and thus the above estimate \eqref{eq:Estim-t2-B} yields finally
\begin{equation}\label{eq:Estim-t-phi}
|t| \leq {\lambda_{k} \beta_{k}\over |\omega|} + \|B\|^{1/2} (\beta_{k}\lambda_{k})^{1/2}(\|v\| + \|z\|).
\end{equation}
Using this together with \eqref{eq:Estim-v} and \eqref{eq:Estim-z}, we infer that
$$\|f\| \leq {\lambda_{k} \beta_{k}\over |\omega|} + (1 + (\|B\|\beta_{k}\lambda_{k})^{1/2})(\alpha_{k-1}(\omega) + \alpha_{k+1}(\omega))(1+(|\omega|\|B\|\|f\|)^{1/2}). $$
From this, upon using Young's inequality $\alpha\beta \leq \epsilon\alpha^2/2 + \epsilon^{-1}\beta^2/2$ on the right hand side, with $\alpha := \|f\|^{1/2}$ and $\beta$ the terms which are factor of $\|f\|^{1/2}$, it is not difficult to choose $\epsilon > 0$ appropriately and obtain \eqref{eq:Estim-LL-omega}, and thus the proof of Proposition \ref{lem:Estim-1} is complete.
\qed
\medskip

\begin{remark}\label{rem:Estim-12} When $B : H_{0} \into H_{0}$ is bounded, the estimate of Proposition \ref{lem:Estim-1} can be improved, but the improvement does not seem fundamental in an abstract result such as the one we present here. Instead, for instance when one is concerned with a wave equation where $B\partial_{t}u := 1_{\omega}\partial_{t}u$, in specific problems one may find better estimates using the local structure of the operator $B$. \qed
\end{remark}

In the next lemma we give a better estimate when, in equation \eqref{eq:u0}, the data $g$ belongs to $H_{0}$ or to $H_{1}$.

\begin{lemma}\label{lem:Estim-1-H0}
Assume that $B$ satisfies \eqref{eq:Cnd-Op-B-1} and \eqref{eq:Cnd-Op-B-2}, $\omega \in {\Bbb R}^*$ and $g\in H_{0}$ be given. Then, the operator $L_{\omega}$ being given by \eqref{eq:L-omega} and with the notations of Proposition \ref{lem:Estim-1}, the solution $u_{0}\in H_{1}$  of $L_{\omega}u_{0} = g$ satisfies
\begin{equation}\label{eq:Estim-H0}
\|u_{0}\| \leq \left({2c(\omega) \over \omega \sqrt{\lambda_{1}}} + {3 \over \omega^2}\right)\, \|g\|.
\end{equation}
Also, for any $\omega \in {\Bbb R}^*$ any $v\in H_{1}$ we have
\begin{equation}\label{eq:Estim-H1}
\|L_{\omega}^{-1}\left(B - {\rm i}\,\omega I \right)v\|_{H_{1}}  \leq  {1 + c(\omega ) \over |\omega|}\, \|v\|_{H_{1}}.
\end{equation}
\end{lemma}

\proof When $g \in H_{0}$, computing $\langle L_{\omega}u_{0},u_{0}\rangle = (u_{0}|g)$ and taking the imaginary part yields
\begin{equation}\label{eq:B-u0-H0}
|\omega|\,\langle Bu_{0},u_{0} \rangle \leq \|g\|\, \|u_{0}\|.
\end{equation}
Then, using Proposition \ref{lem:Estim-1}, we have
\begin{eqnarray*}
\omega^2\|u_{0}\|^2 &= \|L^{1/2}u_{0}\|^2 -{\rm i}\,\omega\langle Bu_{0},u_{0} \rangle - \langle g,u_{0}\rangle \\
&\leq c(\omega)^2\|g\|_{H_{-1}}^2 + 2 \|g\|\|u_{0}\|.
\end{eqnarray*}
From this, and the fact that $\lambda_{1}\|g\|_{H_{-1}}^2 \leq \|g\|^2$ one easily conclude that
$$\omega^2 \|u_{0}\|^2 \leq \left({2c(\omega)^2 \over \lambda_{1}} + {8 \over \omega^2}\right)\|g\|^2.$$
In order to see that \eqref{eq:Estim-H1} holds, it is sufficient to observe that
$$ L_{\omega}^{-1} \left(B - {\rm i}\,\omega I \right) = ({\rm i}\,\omega)^{-1}L_{\omega}^{-1}(L-L_{\omega}) = ({\rm i}\,\omega)^{-1}\left(L_{\omega}^{-1}L- I\right)\, ,$$
and using once more Proposition \ref{lem:Estim-1}, the proof of the Lemma is complete.
\qed

\medskip

We can now prove that ${\rm i}\,{\Bbb R} \subset \rho(A)$, the resolvent set of $A$.

\begin{lemma}\label{lem:Resolvent-A}
Assume that the operator $B$ satisfies  conditions  \eqref{eq:Cnd-Op-B-1} and \eqref{eq:Cnd-Op-B-2}. Then ${\rm i}\,{\Bbb R} \subset \rho(A)$.
\end{lemma}

\proof It is clear that we may assume $|\omega| > 0$, since the case $\omega = 0$ is already treated by Lemma \ref{lem:A-invertible}.

In order to see that $\lambda = {\rm i}\,\omega $  belongs to $\rho(A)$ for any $\omega\in{\Bbb R}^*$, we begin by showing that
$$N(A -\lambda I) = N(A^* - \overline{\lambda}I) = \{0\}.$$
Indeed if $u\in D(A)$ and $Au - {\rm i}\, \omega\, u = 0$, then we have $u_{1} = -{\rm i \omega}\, u_{0}$ and
$$L u_{0} -{\rm i}\,\omega\, Bu_{0} -\omega^2 u_{0} = 0\, ,$$
and by Proposition \ref{lem:Estim-1} we know that $u_{0}=0$, and thus the $N(A-{\rm  i}\omega I)=\{0\}$.
In the same way, one may see that $N(A^* + {\rm i}\omega I) = \{0\}$.
\medskip

Next we show that both $R(A - {\rm i}\omega I)$ and $R(A^* +{\rm i}\omega I)$ are closed. Since it is clearly sufficient to prove the former property, let a sequence $(u_{n})_{n \geq 1} = \{(u_{0n}, u_{1n})\}_{n \geq 1}$ in $D(A)$ be so that
$$f_{n}=(f_{0n}, f_{1n}) := Au_{n} - {\rm i}\,\omega\, u_{n}\to f = (f_{0}, f_{1})  \quad\mbox{in }\, {\Bbb H}.$$
In particular we have
$$-u_{1n} - {\rm i}\,\omega\, u_{0n} = f_{0n} \to f_{0}\quad\mbox{in }\, H_{1}.$$
 Reporting the expression of $u_{1n} = - f_{0n} - {\rm i}\,\omega\, u_{0n}$ into the second component of $Au_{n}$, upon setting
$$g_{n} := f_{1n} + Bf_{0n} - {\rm i}\,\omega\, f_{0n}\, ,$$
and $g := f_{1} + Bf_{0} - {\rm i}\,\omega\, f_{0}$,
we have clearly $g_{n} \to g$ in $H_{-1}$ and
\begin{equation}\label{eq:u0n}
L u_{0n} - {\rm i}\,\omega\, B u_{0n} - \omega^2 u_{0n} = g_{n}.
\end{equation}
Using Proposition \ref{lem:Estim-1}, we know that $L  - {\rm i}\,\omega\, B   - \omega^2  I$ has a bounded inverse, and thus $u_{0n}\to u_{0}$  in $H_{1}$, where $u_{0}$ is the unique solution of
$$L u_{0} - {\rm i}\,\omega\, B u_{0} - \omega^2 u_{0} = g.$$
It is clear that this shows that $u_{n}\to u :=(u_{0},u_{1})$, where $u_{1}=- {\rm i}\,\omega\, u_{0} - f_{0}$. Thus $R(A - {\rm i}\, \omega\, I)$ is closed, in fact $R(A - {\rm i}\, \omega\, I) = {\Bbb H}$ and $(A - {\rm i}\, \omega\, I)^{-1}$ is bounded. \qed

\begin{proposition}\label{lem:Resolvent-A-1}
Assume that the operator $B$ satisfies  conditions  \eqref{eq:Cnd-Op-B-1} and \eqref{eq:Cnd-Op-B-2}. Then there exists a constant $c_{*}> 0$ such that for all $\omega \in {\Bbb R}$ we have
\begin{equation}\label{eq:Estim-Resol}
\|R({\rm i}\,\omega,A)\|\le c_{*}\, c(\omega),
\end{equation}
where
$c(\omega)$ is defined in \eqref{eq:Def-c-omega}.
\end{proposition}

\proof
By Lemma \ref{lem:Resolvent-A} we know that the imaginary axis of the complex plane is contained in the resolvent set of the operator $A$.  For $f=(f_{0},f_{1}) \in {\Bbb H}$, the equation $Au - {\rm i}\,\omega u = f$ can be written as
\begin{equation*} 
\begin{cases}
- u_{1} - {\rm i}\, \omega u_{0} = f_{0} \in H_{1}, \\
L u_{0} + B u_{1} - {\rm i} \,\omega u_{1} = f_{1}\in H_{0}.
\end{cases}
\end{equation*}
Consequently, it follows that
\begin{eqnarray*}
& u_{0} &= L_{\omega}^{-1}\left(B - {\rm i}\, \omega I\right)f_{0} + L_{\omega}^{-1}f_{1} \, , \\
& u_{1} &= -f_{0} - {\rm i}\,\omega u_{0} .
\end{eqnarray*}
By Proposition \ref{lem:Estim-1} and the estimate \eqref{eq:Estim-H1} of Lemma \ref{lem:Estim-1-H0} we have
\begin{equation*}
\begin{array}{rl}
\|u_{0}\|_{H_{1}} &\leq \|L_{\omega}^{-1}\left(B - {\rm i}\, \omega I\right)f_{0}\|_{H_{1}} + \|L_{\omega}^{-1}f_{1}\|_{H_{1}} \\
&\leq \displaystyle {1 + c(\omega ) \over |\omega|}\|f_{0}\|_{H_{1}} + c(\omega)\|f_{1}\|_{H_{-1}}
\\
& \leq c_{*}\, c(\omega)\,\|f\|_{{\Bbb H}}\, ,
\end{array}
\end{equation*}
for some constant $c_{*} >0$ independent of $\omega$ and $f$.

On the other hand, using \eqref{eq:Estim-H0} we have, again for some constant $c_{*}$ independent of $f_{1}$ and $|\omega| \geq 1$ 
$$|\omega|\, \|L_{\omega}^{-1}f_{1}\| \leq c_{*}\,c(\omega)\, \|f_{1}\|,$$
and thus, since $u_{1} = -f_{0} - {\rm i}\,\omega u_{0}$,
\begin{equation*}
\begin{array}{rl}
\|u_{1}\| &\leq \|f_{0}\| + |\omega|\, \|L_{\omega}^{-1}\left(B - {\rm i}\, \omega I\right)f_{0}\| + |\omega|\, \|L_{\omega}^{-1}f_{1}\| \\
& \leq c_{*}\, c(\omega)\,\|f\|_{{\Bbb H}}\, ,
\end{array}
\end{equation*}
for some appropriate constant $c_{*}$ independent of $\omega$ and $f$.
\qed
\medskip

We are now in a position to prove our main abstract result.

\noindent{\bf Proof of Theorem \ref{lem:Main-1}. }
Take $\omega \in {\Bbb R}$. In order to prove our claim, using Theorem \ref{lem:Borichev-Tomilov}, it is enough to show that the constant $c(\omega)$ which appears in \eqref{eq:Estim-Resol} has a growth rate of at most $(1 + |\omega|^m)$, where $m$ is given by \eqref{eq:Decay-Abs}.  It is clear that it is sufficient to prove the estimate on $c(\omega)$ when $\omega^2 \geq \lambda_{1}/2$. Therefore, assuming that such is the case, there is an integer $k\geq 1$ such that
$$\lambda_{k-1} < {\lambda_{k-1} + \lambda_{k} \over 2} \leq \omega^2 \leq {\lambda_{k} + \lambda_{k+1} \over 2} < \lambda_{k+1}.$$
Thus, with the notations of Proposition \ref{lem:Estim-1}, we have
\begin{equation}
\alpha_{k-1}(\omega) + \alpha_{k+1}(\omega) \leq {2\lambda_{k-1} \over \lambda_{k} - \lambda_{k-1}}
+ {2\lambda_{k+1} \over \lambda_{k+1} - \lambda_{k}} \leq 2\, c_{0}\, \lambda_{k}^{\gamma_{1}}.
\end{equation}
Now, thanks to the assumption \eqref{eq:Poly-growth-eigen} we have $\lambda_{k+1}\leq \lambda_{k}/\lambda_{*}$ and, when $k\geq 2$, we have also $\lambda_{k-1} \geq \lambda_{*}\lambda_{k}$. From this we may conclude that
$${1 \over 2}(1 + \lambda_{*})\lambda_{k} \leq \omega^2 \leq {1 +\lambda_{*} \over 2\lambda_{*}}\lambda_{k}.$$
Using the expression of $c(\omega)$ given by \eqref{eq:Def-c-omega}, one may find a constant $c_{0}^* > 0$, depending only $\lambda_{*},c_{0}$ and on $\lambda_{1}, c^*$, so that for all $\omega$ with $\omega^2 \geq \lambda_{1}/2$ we have
$$c(\omega) \leq c_{0}^*\,\left(|\omega|^{1+2\gamma_{0}} + (1 + |\omega|^{2(1+\gamma_{0})})\,|\omega|^{4\gamma_{1}}\,(1 + |\omega|) \right).$$
From this, setting $m := 3 + 2\gamma_{0} + 4\gamma_{1}$, it is not difficult to see that one has  $c(\omega)\leq c_{1}^*(1+|\omega|^m)$, at the expense of choosing another constant $c_{1}^*$, and the proof of our Theorem is complete.
\qed
\medskip

In the following sections we give a few  examples of damped wave equations which can be treated according to Theorem \ref{lem:Main-1}.

\section{Wave equations in dimension one}\label{sec:Wave-dim-1}

In this section we give a few applications of Theorem \ref{lem:Main-1} to the case of a class of wave equations, in dimension one, that is a system corresponding to the vibrations of a string. The treatment of such a problem is easier in one dimension than in higher dimensions, due to the fact that on the one hand the multiplicity of each eigenvalue is one, the distance between consecutive eigenvalues is large, and on the other hand the eigenfunctions are explicitely known in some cases, and have appropriate asymptotic behaviour when they are not explicitely known. 

More precisely, without loss of generality, we may assume that $\Omega = (0,\pi)$ and, 
with the notations of the previous section, we set $H_{0} := L^2(0,\pi)$, the scalar product of $f,g \in L^2(0,\pi) = L^2((0,\pi),{\Bbb C})$ being denoted by
$$(f|g) := \int_{0}^{\pi}f(x)\, \overline{g(x)}\,dx,$$
and the associated norm by $\|\cdot\|$. Let $a \in L^\infty(0,\pi)$ be a  positive  function such that for a certain $\alpha_{0} > 0$, we have $a(x) \geq \alpha_{0}$ a.e. in  $(0,\pi)$. Then, two nonnegative functions $b_{1},b_{2} \in L^\infty(0,\pi)$ being given, the system 
\begin{equation}\label{eq:System-21}
\begin{cases}
\partial_{tt}u - \partial_{x}(a\partial_{x} u) + b_{1}\partial_{t}u - \partial_{x}(b_{2}\partial_{x} \partial_{t}u)= 0  & \mbox{ in } (0, \infty) \times (0,\pi), \\
u(t,0) = u(t,\pi) = 0  & \mbox{ on } (0, \infty), \\
u(0,x)  = u_{0}(x) & \mbox{ in } (0,\pi) \\
u_{t}(0,x)  = u_{1}(x) & \mbox{ in } (0,\pi),
\end{cases}
\end{equation}
is a special case of the system \eqref{eq:System-Abs}. We are going to verify that under certain circumstances, we can apply Theorem \ref{lem:Main-1} and obtain a polynomial decay for the energy associated to equation \eqref{eq:System-21}.

First we consider the operator $(L,D(L))$ defined by
\begin{align}
& Lu := - (a(\cdot) u')' \label{eq:Def-L-wave-1} \\
& D(L) := \left\{u \in H^1_{0}(0,\pi) \; ; \; Lu \in L^2(0,\pi)\right\},\label{eq:Def-D-L-wave-1}
\end{align}
which is a selfadjoint, positive operator with a compact resolvent and one has $H_{1} := D(L^{1/2}) = H^1_{0}(0,\pi)$. We shall endow $H^1_{0}(0,\pi)$ with the scalar product
$$(u|v)_{H^1_{0}} := \int_{0}^\pi u'(x) \cdot \overline{v'(x)}\, dx,$$
and its associated norm $u \mapsto \|u'\|$ (the resulting topology is equivalent to that resulting from the equivalent Hilbertian norm $u\mapsto \|a^{1/2} u'\|$).

For the operator $B$, assuming that the functions $b_{1},b_{2}$ are such that at least one of the conditions \eqref{eq:Cnd-b1} or \eqref{eq:Cnd-b2} is satisfied, we define
\begin{equation}\label{eq:Def-Op-B-1}
B\phi := b_{1}\phi - (b_{2} \phi')'.
\end{equation}
It is easy to verify that the operator $B$ is bounded and selfadjoint from $H_{0}^1(0,\pi)$ into $H^{-1}(0,\pi)$ and that it satisfies conditions \eqref{eq:Cnd-Op-B-1}.

Assume also that $\Omega_{1}$ and $\Omega_{2}$ are given by
\begin{equation}\label{eq:ell-12}
\Omega_{1} = (\ell_{1},\ell_{1}+\delta_{1}),\qquad
\Omega_{2} := (\ell_{2},\ell_{2} + \delta_{2})
\end{equation}  
with $0\leq \ell_{1} < \ell_{1} + \delta_{1} \le \pi$ and  $0\leq \ell_{2} < \ell_{2} + \delta_{2} \leq \pi$. Then we have the following result:

\begin{proposition}\label{lem:String}
Assume that $N = 1$ and let the domains $\Omega_{1},\Omega_{2}$  be as in \eqref{eq:ell-12} with $\delta_{2} > 0$. Let the function $a$ in \eqref{eq:System-21} be of class $C^2([0,\pi])$ and, for $j=1$ or $j=2$, the functions $b_{j} \in L^\infty(0,\pi)$ be such that $b_{j} \geq \epsilon_{j} \geq 0$ on $\Omega_{j}$, where $\epsilon_{j}$ is a constant. Then, if $\epsilon_{2} > 0$, there exists a constant $c_{*} > 0$ such that the energy of the solution of \eqref{eq:System-21} satisfies
\begin{equation}\label{eq:Decay-N-1}
\|\partial_{x}u(t,\cdot)\|^2 + \|\partial_{t}u(t,\cdot)\|^2 \leq c_{*}\, (1+t)^{-2/3}\left[
\|\partial_{x}u_{0}\|^2 + \|u_{1}\|^2\right].
\end{equation}
Also, if $b_{2} \equiv 0$ and $\epsilon_{1} > 0$, there exists a constant $c_{*} > 0$ such that
\begin{equation}\label{eq:Decay-N-1-bis}
\|\partial_{x}u(t,\cdot)\|^2 + \|\partial_{t}u(t,\cdot)\|^2 \leq c_{*}\, (1+t)^{-1/2}\left[
\|\partial_{x}u_{0}\|^2 + \|u_{1}\|^2\right].
\end{equation}
\end{proposition}

\proof Consider first the case $a(x) \equiv 1$. Then, for all integers $k \geq 1$
$$\lambda_k = k^2, \quad\mbox{and}\quad \phi_k =  \sqrt{2/\pi}\sin(kx). $$
One sees immediately that for some constant $c_{*} > 0$ independent of $k$ we have
$$
{\lambda_{k-1} \over \lambda_{k} - \lambda_{k-1}} + {\lambda_{k+1} \over \lambda_{k+1} - \lambda_{k}} \leq k \leq c_{*}\, \lambda_{k}^{1/2} ,$$ 
and thus, with the notations of Theorem \ref{lem:Main-1}, we can take $\gamma_1 =1/2$.

On the other hand, when $\epsilon_{2} > 0$, one checks easily that for some constant $c$ independent of $k$ we have
\begin{equation}
\langle B\phi_k,\phi_k\rangle  \geq 
\epsilon_{2} \int_{\ell_{2}}^{\ell_{2}+ \delta_{2}} |\phi_{k}'(x)|^2dx = {2\epsilon_{2} k^2 \over \pi}
 \int_{\ell_{2}}^{\ell_{2} + \delta_{2}}\!\!\! \cos^2(kx)\,dx \geq c\, k^2.
\end{equation} 
Therefore for some other constant $c_{*}$ independent of $k$, we have $\beta_k \leq c_{*} \lambda_{k}^{-1}$, and thus we can take $\gamma_{0} := -1$.

Finally we have $m  = 3 + 2\gamma_{0} + 4\gamma_{1} = 3$ and, according to Theorem \ref{lem:Borichev-Tomilov}, the semigroup decays polynomially with rate $1/3$, that is the decay estimate for the energy is given by \eqref{eq:Decay-N-1}.
\medskip

When $b_{2} \equiv 0$ and $\epsilon_{1} > 0$, then the only damping comes from the term involving $b_{1}$ and in this case
\begin{equation}
\langle B\phi_k,\phi_k\rangle  \geq 
\epsilon_{1} \int_{\ell_{1}}^{\ell_{1}+ \delta_{1}} |\phi_{k}(x)|^2dx = {2\epsilon_{1} \over \pi}
 \int_{\ell_{1}}^{\ell_{1} + \delta_{1}}\!\!\! \sin^2(kx)\,dx \geq c.
\end{equation} 
Thus $\beta_{k} \leq c_{*}$, and we can take $\gamma_{0} := 0$. From this we infer that $m = 3 + 2\gamma_{0} + 2\gamma_{1} = 4$, which means that \eqref{eq:Decay-N-1-bis} holds.
\medskip

When $a$ is not identically equal to $1$, it is known that there exist two positive constants $C_{1}, C_{2}$ and a sequence of real numbers $(c_{k})_{k \geq 1}$, satisfying $\sum_{k\geq 1}|c_{k}|^2 < \infty$, such that as $k \to \infty$ the eigenvalues $\lambda_{k}$ and eigenfunctions $\phi_{k}$ satisfy, uniformly in $x$, 
\begin{align}
&\lambda_{k} = \ell^2 k^2 + C_{1} + c_{k}, \label{eq:Expan-1}\\
&\phi_{k}(x) = C_{2}\, a(x)^{-1/4} \sin(k\xi(x)) + O(k^{-1}), \label{eq:Expan-2}\\
&\phi_{k}'(x) = C_{2}\,a(x)^{-3/4}\, k\, \cos(k\xi(x)) + O(1),\label{eq:Expan-3}
\end{align}
where
$$\ell := \int_{0}^\pi a(y)^{-1/2}\,dy, \quad \xi(x):= {\pi \over \ell}\int_{0}^x a(y)^{-1/2}\,dy,\quad\mbox{and}\quad
\sum_{k\geq 1}|c_{k}|^2 <\infty .$$
These formulas are obtained through the Liouville transformation, and we do not give the details of their computations, since we can refer to J. P\"oschel \& E. Trubowitz \cite{Trubowitz}, or A.~Kirsch \cite[Chapter 4]{Kirsch}. Indeed in the latter reference, in Theorem 4.11, the result is stated for the Dirichlet eigenvalue problem $-\phi'' + q \phi = \lambda \phi$, but one may show that after an appropriate change of variable and unknown function, described in the introduction of chapter 4, on pages 121--122 of this reference, one can prove the formulas given above, which are of interest in our case.

Now, according to the definition of $x\mapsto \xi(x)$, making a change of variable in the first integral below, one has 
$$\int_{\ell_{2}}^{\ell_{2}+\delta_{2}}a(x)^{-3/2}\cos^2(k\xi(x))\,dx = {\ell \over \pi}\int_{\xi(\ell_{2})}^{\xi(\ell_{2}+\delta_{2})} a(x(\xi))^{-1}\cos^2(k\xi)\,d\xi,$$
so that on a close examination of the asymptotic expansions \eqref{eq:Expan-1}--\eqref{eq:Expan-3},  one is convinced that the same values for the exponents $\gamma_{0}$ and $\gamma_{1}$ of Theorem \ref{lem:Main-1} can be obtained, and the proof of the Proposition is complete.
\qed

\begin{remark}
When $a\equiv 1$, a great number of results exist in the literature. In particular, assuming that $b_{1} \equiv 0$ and $b_{2} := 1_{\Omega_{2}}$, Z.~Liu and B.P.~Rao \cite{Liu-Rao-2005}, M.~Alves \& al.~\cite{Alves-Rivera} have shown that the semigroup has a decay rate of $(1+t)^{-2}$, thus the energy decays with the rate $(1+t)^{-4}$, and that this decay rate is optimal. However the cases in which $a\not\equiv 1$, or $b_{1} \geq \epsilon_{1} >0$ on $\Omega_{0}$, are not covered by these authors, while the method we present here can handle such cases, at the cost of not establishing an optimal decay in simpler cases. \qed
\end{remark}

\begin{remark}\label{rem:SL-general}
As a matter of fact the same decay rate of the energy, with the same exponent number $m = 3$, holds for a wave equation of the form
$$\rho(x)\partial_{tt}u - \partial_{x}(a(x)\partial_{x}u) + q(x)u + b_{0}(x)\partial_{t} u - \partial_{x}(b_{1}(x)\partial_{xt}u) = 0.$$
In such a case, the operator $L$ will be given by
\begin{equation}\label{eq:SL-general}
Lu := -\rho(x)^{-1}(a(x)u')' + \rho(x)^{-1}q(x)u,
\end{equation}
where $\rho$ and $a$ belong to $C^2([0,\pi])$ and $\min(\rho(x),a(x)) \geq \epsilon_{0} > 0$, while the potential $q\in C([0,\pi])$ is such that the least eigenvalue $\lambda_{1}$ of the problem
$$-(a(x)\phi')' + q\phi = \lambda \rho(x)\phi, \qquad \phi(0) = \phi(\pi) = 0,$$
verifies $\lambda_{1} > 0$ (in fact any other boundary conditions, such as Neumann, or Fourier conditions, ensuring that the first eigenvalue $\lambda_{1} > 0$, can be handled, with the same decay rate for the corresponding wave equation). Indeed, such an operator $L$ is selfadjoint in the weighted Lebesgue space $L^2(0,\pi, \rho(x)dx)$, and it is known that (see for instance A.~Kirsch \cite{Kirsch}, as cited above) an expansion of the form \eqref{eq:Expan-1}--\eqref{eq:Expan-3} holds in this case for the eigenvalues and eigenfunctions of $L$, with the only difference that in \eqref{eq:Expan-2} the function $a(x)^{-1/4}$  should be replaced by $a(x)^{-1/4}\rho(x)^{-1/4}$, and in \eqref{eq:Expan-3}  the function $a(x)^{-3/4}$  should be replaced by $a(x)^{-3/4}\rho(x)^{-1/4}$.
\qed
\end{remark}
\medskip

\begin{remark}\label{rem:a-discont}
It is noteworthy to observe that the assumption $a\in C^2([0,\pi])$ of Proposition \ref{lem:String}, as well as the condition $\rho\in C^2([0,\pi])$ in Remark \ref{rem:SL-general}, are needed in order to apply the general result which ensures the precise asymptotics \eqref{eq:Expan-1}--\eqref{eq:Expan-3}.
We are not aware of any result analogous to the precise expansion properties \eqref{eq:Expan-1}--\eqref{eq:Expan-3} in the general case where $a,\rho$ are only in $L^\infty(0,\pi)$.

However, in some cases in which the coefficients $a$ and $\rho$ are not smooth, it is nevertheless possible to show that the behaviour of the eigenvalues $\lambda_{k}$ and eigenfunctions $\phi_{k}$ resembles those of the Laplace operator with Dirichlet boundary conditions on $(0,\pi)$.
Such an example may be given by coefficients having a finite number of discontinuites, such as step functions, for which explicit calculation of $\lambda_{k}$ and $\phi_{k}$ is possible. For instance consider $\rho(x) \equiv 1$ and $a\in L^\infty(0,\pi)$ the piecewise constant function given by
$$a(x) := 1_{(0,\pi/2)}(x) + 4\times 1_{(\pi/2,\pi)}(x),$$
where for a set $A$ the function $1_{A}$ denotes the characteristic function of $A$.
Then a simple, but perhaps somewhat dull, if not tedious, calculation shows that the eigenvalues and eigenfunctions solutions to
$$-(a(x)\phi_{k}'(x))' = \lambda_{k} \phi_{k}(x), \qquad \phi_{k}(0) = \phi_{k}(\pi) = 0,$$
are given by 
$$\left\{(\lambda_{k},\phi_{k}) \; ; \; k \geq 1\right\} = \left\{(\mu_{1,m},\phi_{1,m}) \; ; \; m \in {\Bbb N}^*\right\} \cup
\left\{(\mu_{2,n},\phi_{2,n}) \; ; \; n \in {\Bbb Z}\right\},
 $$
where the sequences
$(\mu_{1,m},\phi_{1,m})_{m \geq 1}$ and $(\mu_{2,n},\phi_{2,n})_{n \in {\Bbb Z}}$ are defined as follows. For $m \geq 1$ integer, $\mu_{1,m} := 16m^2$ and (up to a multiplicative normalizing constant independent of $m$)
$$
\phi_{1,m}(x) = \sin(4mx)1_{(0,\pi/2)}(x) + 2\times (-1)^{m}\sin(2mx)\, 1_{(\pi/2,\pi)}(x). $$
Also, for $n\in {\Bbb Z}$, the sequence $\lambda_{2,n}$ is given by
$$\mu_{2,n} := 16 \left(n + {\arctan(\sqrt{2}) \over \pi}\right)^2,$$
and (again up to a multiplicative normalizing constant independent of $n$)
$$
\phi_{2,n}(x) := \sin(\sqrt{\mu_{2,n}}\,x)\, 1_{(0,\pi/2)} + 
(-1)^n\,{2\sqrt{3} \over 3}  \, \sin\left({\sqrt{\mu_{2,n}}\,(\pi -x) \over 2}\right)1_{(\pi/2,\pi)} .
$$
Now it is clear that proceeding as in the proof of Proposition \ref{lem:String}, when $\epsilon_{2} > 0$, one can infer that the exponent $m$ in Theorem \ref{lem:Main-1} can be taken as $m=3$, so that the decay rate of the energy is at least $(1+t)^{-2/3}$. Analogously when $b_{2} \equiv 0$ and $\epsilon_{1} > 0$, then one can take $m = 4$ and the energy decays at least with the rate $(1+t)^{-1/2}$.
\qed
\end{remark}

\section{Wave equations in higher dimensions}\label{sec:Wave-dim-2}

Our next example of a damped wave equation for which a decay rate of the energy can be proven using Theorem~\ref{lem:Main-1}, with an explicitly computed rate of decay, 
for the solution of 
\begin{equation}\label{eq:System-22}
\begin{cases}
\partial_{tt}u - \Delta u + b_{1}\partial_{t}u - {\rm div}(b_{2}\nabla \partial_{t}u)= 0  & \mbox{ in } (0, \infty) \times \Omega, \\
u(t,\sigma)  = 0  & \mbox{ on } (0, \infty)\times \partial\Omega, \\
u(0,x)  = u_{0}(x) & \mbox{ in } \Omega \\
u_{t}(0,x)  = u_{1}(x) & \mbox{ in } \Omega,
\end{cases}
\end{equation}
where $b_{j} \in L^\infty(\Omega)$ are nonnegative functions.

As we shall se below, in order to find the adequate exponents $\gamma_{0}$ and $\gamma_{1}$ which are used in Theorem \ref{lem:Main-1}, one has to carry out a precise analysis of the behaviour of the eigenvalues and eigenfunctions of the underlying operator. As far as the Laplace operator is concerned, in a few cases one can perform this analysis, but even in those cases one sees that the exponents $\gamma_{0}$ and $\gamma_{1}$ depend in avery subtle way on the domain $\Omega$.

The following lemma takes of the condition \eqref{eq:Poly-growth-eigen} in the cases we study in this section.

\begin{lemma}\label{lem:Lim-1}
Let $\Omega := (0,L_{1})\times\cdots\times(0,L_{N}) \subset {\Bbb R}^N$ where $L_{j} > 0$ for $1\leq j \leq N$. Denoting by $(\lambda_{k})_{k\geq 1}$ the eigenvalues of the Laplacian operator with Dirichlet boundary conditions on $\Omega$, we have
$$\lim_{k\to\infty} {\lambda_{k+1} \over \lambda_{k}} = 1.$$
\end{lemma}

\begin{proof} 
The eigenvalues of the operator $L$ defined by $Lu:= -\Delta u$ with 
$$D(L) := \left\{u\in H^1_{0}(\Omega) \; ; \; \Delta u \in L^2(\Omega)\right\},$$ 
are given by
$${\widetilde \lambda}_{n} := {n_{1}^2 \pi^2 \over L_{1}^2} + \cdots + {n_{N}^2 \pi^2 \over L_{N}^2},\qquad\mbox{for }\, n = (n_{1},\ldots,n_{N})\in ({\Bbb N}^*)^N.$$
As before we denote by $(\lambda_{k})_{k\geq 1}$ the sequence of eigenvalues obtained upon reordering this family $({\widetilde \lambda}_{n})_{n}$, with the convention that each eigenvalue has multiplicity $m_{k} \geq 1$ and $\lambda_{k} < \lambda_{k+1}$.

If all the eigenvalues $\lambda_{k}$ were simple, we could use Weyl's formula, asserting that $\lambda_{k} \sim c_{*} k^{2/N}$ as $k\to \infty$ (see W.~Arendt \& al. \cite{ArendtW-Nittka-2009}, H. Weyl \cite{Weyl-1,Weyl-2}) where $c_{*}>0$ is a constant depending only on $L_{1},\ldots,L_{N}$. However, here we have made the convention that $\lambda_{k} < \lambda_{k+1}$, each eigenvalue $\lambda_{k}$ having multiplicity $m_{k} \geq 1$, and thus if one has no information on $m_{k}$, one cannot use Weyl's formula.

However in general the eigenvalues are not simple, implying that we cannot use directly Weyl's formula. Nevertheless the proof of the Lemma can be done in an elementary way: consider a sequence of integers $k_{j} \to +\infty$ as $j\to+\infty$. Then there exists a sequence of $N$-tuples of integers $n_{j} \in ({\Bbb N}*)^N$ such that 
$$\lambda_{k_{j}} = {n_{j1}^2 \pi^2 \over L_{1}^2} + \cdots + {n_{jN}^2 \pi^2 \over L_{N}^2}.$$
It is clear that necessarily there exists $\ell_{j} > k_{j}$ such that
$$\lambda_{\ell_{j}} = {(n_{j1} + 1)^2 \pi^2 \over L_{1}^2} + \cdots + {(n_{jN} + 1)^2 \pi^2 \over L_{N}^2},$$
and thus $\lambda_{k_{j}+1} \leq \lambda_{\ell_{j}}$. Thus we have
$$1 \leq {\lambda_{k_{j} + 1} \over \lambda_{k_{j}}} \leq 
\left(
{(n_{j1} + 1)^2\over L_{1}^2} + \cdots + 
{(n_{jN} + 1)^2 \over L_{N}^2}
\right)
\left(
{n_{j1}^2 \over L_{1}^2} + \cdots + 
{n_{jN}^2 \over L_{N}^2}\right)^{-1}.
$$
Since $k_{j} \to \infty$, we have $\max\left\{n_{ji}\; ; \; 1\leq i \leq N\right\} \to \infty$, and thus
$$\lim_{j\to\infty} 
\left(
{(n_{j1} + 1)^2\over L_{1}^2} + \cdots + 
{(n_{jN} + 1)^2 \over L_{N}^2}
\right)
\left(
{n_{j1}^2 \over L_{1}^2} + \cdots + 
{n_{jN}^2 \over L_{N}^2}\right)^{-1} = 1,$$
and the proof of the Lemma is complete.\qed
\end{proof}
\medskip
To illustrate how Theorem \ref{lem:Main-1} can be used, first we investigate the case of dimension $N = 2$ with a choice of the domains $\Omega,\Omega_{1}$ and $\Omega_{2}$ as follows
\begin{equation}\label{eq:Def-Omega-j}
\begin{cases}
\Omega := (0,\pi)\times(0,\pi),\\
\Omega_{1} := (\ell_{1},\ell_{1} + \delta_{1})\times (0,\pi), \\
\Omega_{2} := (\ell_{2},\ell_{2} + \delta_{2})\times (0,\pi),
\end{cases}
\end{equation}
where, for $j=1$ and $j=2$, it is assumed that $0\leq \ell_{j} < \ell_{j} + \delta_{j} \leq \pi$.

As an inspection of the proof of the following proposition shows, the exact same result holds when  one of the sets $\Omega_{1}$ or $\Omega_{2}$ is a horizontal strip, and also for any dimension $N \geq 2$ with $\Omega := (0,\pi)^N$ while the {\it damping subdomains\/} $\Omega_{1},\Omega_{2}$ are narrow strips of the above type, parallel to one of the axis and touching the boundary of $\Omega$. As we have mentioned before, one can also consider the case of an operator such as $Lu := -\Delta u$ with  boundary conditions which ensure that $L$ is self-adjoint and its least eigenvalue is positive (for instance mixed Neumann and Dirichlet boundary conditions, or of Fourier type, also called Robin type boundary condition). However, for the sake of clarity of exposition, we present the result, and its proof, only for the case $N=2$ and Dirichlet boundary conditions.

Then we can state the following:

\begin{proposition}\label{lem:Membrane-square}
Assume that $N = 2$ and the domains $\Omega,\Omega_{1},\Omega_{2}$  are as in \eqref{eq:Def-Omega-j}. For $j=1$ or $j=2$, let the  functions $b_{j} \in L^\infty(\Omega)$ be such that $b_{j} \geq \epsilon_{j} \geq 0$ on $\Omega_{j}$, where $\epsilon_{j}$ is a constant. Then, when $\epsilon_{2} > 0$, there exists a constant $c_{*} > 0$ such that the energy of the solution of \eqref{eq:System-22} satisfies
\begin{equation}\label{eq:Decay-N-2}
\|\nabla u(t,\cdot)\|^2 + \|\partial_{t}u(t,\cdot)\|^2 \leq c_{*}\, (1+t)^{-2/5}\left[
\|\nabla u_{0}\|^2 + \|u_{1}\|^2\right].
\end{equation}
When $\epsilon_{1} > 0$ and $b_{2} \equiv 0$, one has
\begin{equation}\label{eq:Decay-N-22}
\|\nabla u(t,\cdot)\|^2 + \|\partial_{t}u(t,\cdot)\|^2 \leq c_{*}\, (1+t)^{-2/7}\left[
\|\nabla u_{0}\|^2 + \|u_{1}\|^2\right].
\end{equation}

\end{proposition} 

\proof Setting $Lu:= -\Delta u$ with 
$$D(L) := \left\{u\in H^1_{0}(\Omega) \; ; \; \Delta u \in L^2(\Omega)\right\},$$ 
the eigenvalues and eigenfunctions of the operator $L$ are given by
\begin{equation}\label{eq:Eigen-square}
\widetilde{\lambda}_{n} := n_{1}^2 + n_{2}^2, \quad
\phi_{n}(x) := {2 \over \pi}\sin(n_{1}x_{1})\sin(n_{2}x_{2}), \quad
\mbox{for }\, n \in {\Bbb N}^* \times {\Bbb N}^*.
\end{equation}
Rearranging these eigenvalues $\widetilde{\lambda}_{n}$ in an increasing order, we denote them by $(\lambda_{k})_{k\geq 1}$, the multiplicity of each $\lambda_{k}$ being 
\begin{equation}\label{eq:Def-J-k}
m_{k} :={\rm card}(J_{k}), \qquad\mbox{where }\; J_{k} := \left\{n \in {\Bbb N}^*\times {\Bbb N}^* \; ; \; n_{1}^2 + n_{2}^2 = \lambda_{k}\right\}.
\end{equation}
To begin with the verification of the conditions of Theorem \ref{lem:Main-1}, we recall that Lemma \ref{lem:Lim-1} ensures that we have $\lim_{k\to \infty} \lambda_{k}/\lambda_{k+1} = 1$, and thus condition \eqref{eq:Poly-growth-eigen} is satisfied.

Observe also that each $\lambda_{k}$ being an integer, we have $\lambda_{k+1} - \lambda_{k} \geq 1$, and thus there exists a constant $c_{*} > 0$ such that for all $k \geq 1$ we have
$${\lambda_{k-1} \over \lambda_{k} - \lambda_{k-1}} + {\lambda_{k+1} \over \lambda_{k+1} - \lambda_{k}} \leq c_{*} \, \lambda_{k},$$
and therefore condition \eqref{eq:Cnd-alpha-k} is also satisfied with $\gamma_{1} = 1$.
\medskip

When $\epsilon_{2} > 0$, in order to verify condition \eqref{eq:Cnd-beta-k}, it is sufficient to show that there exist $\gamma_{0} \in {\Bbb R}$ and some constant $c_{*} > 0$, such that for any $k \geq 1$ and any $\phi \in N(L-\lambda_{k}I)$ with $\|\phi\| = 1$ we have
\begin{equation}\label{eq:beta-2}
\int_{\Omega_{2}}|\nabla\phi(x)|^2dx \geq c_{*} \lambda_{k}^{-\gamma_{0}}.
\end{equation}
Since the family $(\phi_{n})_{n\in J_{k}}$ is a Hilbert basis of the finite dimensional space $N(L- \lambda_{k}I)$, we have
\begin{equation}\label{eq:phi-k}
\phi \in N(L-\lambda_{k}I),\; \|\phi\| = 1 
\iff 
\phi = \sum_{n\in J_{k}}c_{n}\phi_{n}\mbox{ with } \sum_{n\in J_{k}} |c_{n}|^2 = 1.
\end{equation}
Thus we have
\begin{align}\label{eq:grad-sum}
\int_{\Omega_{2}}\!\!|\nabla\phi(x)|^2dx &= \sum_{n\in J_{k}}|c_{n}|^2\int_{\Omega_{2}}\!\!|\nabla\phi_{n}(x)|^2dx \nonumber\\
& \qquad + \sum_{n,m \in J_{k}\atop n\neq m}c_{n}\overline{c_{m}}\int_{\Omega_{2}}\!\! \nabla\phi_{n}(x)\cdot\nabla\phi_{m}(x)\, dx. 
\end{align}
Now it is clear that we have
$$\int_{\Omega_{2}}|\partial_{1}\phi_{n}(x)|^2dx = {4n_{1}^2 \over \pi^2} \int_{0}^\pi \int_{\ell_{2}}^{\ell_{2}+\delta_{2}}\cos^2(n_{1}x_{1})\,\sin^2(n_{2}x_{2})\,dx_{1}dx_{2},$$
which yields 
$$\int_{\Omega_{2}}|\partial_{1}\phi_{n}(x)|^2dx = {2n_{1}^2 \over \pi} \int_{\ell_{2}}^{\ell_{2}+\delta_{2}}\cos^2(n_{1}x_{1})\, dx_{1} .$$
Analogously, we have
$$\int_{\Omega_{2}}|\partial_{2}\phi_{n}(x)|^2dx = {2n_{2}^2 \over \pi} \int_{\ell_{2}}^{\ell_{2}+\delta_{2}}\sin^2(n_{1}x_{1})\, dx_{1} ,$$
and thus one can find a constant $c_{*} > 0$ such that for all $k \geq 1$ and all $n \in J_{k}$, we have
\begin{equation}\label{eq:grad-nn}
\int_{\Omega_{2}}|\nabla\phi_{n}(x)|^2dx \geq c_{*}\,\delta_{2}\,(n_{1}^2 + n_{2}^2) = c_{*}\,\delta_{2}\,\lambda_{k}.
\end{equation}
Regarding the second sum in \eqref{eq:grad-sum}, taking $n,m\in J_{k}$ and $n \neq m$, we observe that since $n_{1}^2 + n_{2}^2 = m_{1}^2 + m_{2}^2$, we have necessarily $n_{2}\neq m_{2}$ and therefore,
\begin{align*}
&\int_{\Omega_{2}}\partial_{1}\phi_{n}(x)\partial_{1}\phi_{m}(x)\,dx =& \\
&{4n_{1}m_{1} \over \pi^2}\int_{0}^\pi\int_{\ell_{2}}^{\ell_{2}+\delta_{2}}\cos(n_{1}x_{1})\cos(m_{1}x_{1})\, dx_{1}\sin(n_{2}x_{2})\sin(m_{2}x_{2})\, dx_{2} = 0.
\end{align*}
In the same manner one may see that
\begin{align*}
&\int_{\Omega_{2}}\partial_{2}\phi_{n}(x)\partial_{2}\phi_{m}(x)\,dx =& \\
&{4n_{2}m_{2} \over \pi^2}\int_{0}^\pi\int_{\ell_{2}}^{\ell_{2}+\delta_{2}}\sin(n_{1}x_{1})\sin(m_{1}x_{1})\, dx_{1}\cos(n_{2}x_{2})\cos(m_{2}x_{2})\, dx_{2} = 0.
\end{align*}
Finally one sees that for all $n,m\in J_{k}$ such that $n\neq m$ we have
$$\int_{\Omega_{2}}\nabla\phi_{n}(x)\cdot\nabla\phi_{m}(x)\, dx = 0,$$
so that reporting this and \eqref{eq:grad-nn} into \eqref{eq:grad-sum} we have, for all $\phi\in N(L- \lambda_{k}I)$ with $\|\phi\| = 1$,
$$\int_{\Omega_{2}}|\nabla\phi(x)|^2dx \geq c_{*}\, \delta_{1}\,\lambda_{k},$$
which means that, when $\epsilon_{2} > 0$, the inequality \eqref{eq:beta-2}, and thus \eqref{eq:Cnd-beta-k}, is satisfied with $\gamma_{0} = -1$.

Therefore, when $\epsilon_{2} > 0$ we have $m := 3 + 2\gamma_{0} + 4\gamma_{1} = 5$, and \eqref{eq:Decay-N-2} holds.

\medskip
When $b_{2} \equiv 0$ and $\epsilon_{1} > 0$, proceeding as above, one checks easily that using \eqref{eq:phi-k} there exists a constant $c_{*} > 0$ such that for $\phi \in N(L - \lambda_{k}I)$ and $\|\phi\| = 1$ we have
$$\int_{\Omega_{1}}|\phi(x)|^2\,dx \geq c_{*}.$$
Thus we may take $\gamma_{0} = 0$, so that $m = 3 + 2\gamma_{0} + 4\gamma_{1} = 7$, yielding \eqref{eq:Decay-N-22}, and the proof of our claim is complete.
\qed
\medskip

As a matter of fact, one sees that in order to establish a decay result for other domains $\Omega \subset {\Bbb R}^N$ and $N \geq 2$, there are two issues which should  be inspected carefully: the first one is an estimate of $(\lambda_{k+1} - \lambda_{k})$ from below (comparing it with a power of $\lambda_{k}$) and this is related to the concentration properties of the eigenvalues as $k\to\infty$. The second issue is to obtain an estimate of the local norm of an eigenfunction $\phi$ on $\Omega_{1}$, or that of $\nabla\phi$ on $\Omega_{2}$, and this is related to the concentration properties of the eigenfunctions of the Laplacian. 

Regarding the first issue we shall use the following lemma for the special case of two dimenional rectangles.

\begin{lemma}\label{lem:Gap-1}
Let $\xi > 0$ be a real number and for $n\in ({\Bbb N}^*)^2$ denote $\mu_{n}(\xi) := n_{1}^2 + \xi n_{2}^2$ and
$$\delta(\xi) := \inf\left\{|\mu_{n} - \mu_{m}| \; ; \; n,m \in ({\Bbb N}^*)^2,\; \mu_{n}\neq\mu_{m}\right\}.$$
Then we have $\delta(\xi) \geq 1/q$ if $\xi = p/q$ where $p,q \geq 1$ are integers and mutually prime, while $\delta(\xi) = 0$ if $\xi$ is irrational.
\end{lemma}

\proof If $\xi = p/q$ for two mutually prime integers $p,q\geq 1$, then we have
$$|\mu_{n} - \mu_{m}| = {1 \over q}\left|q(n_{1}^2 - m_{1}^2) + p(n_{2}^2 - m_{2}^2)\right| \geq {1 \over q}\, ,$$
because $q(n_{1}^2 - m_{1}^2) + p(n_{2}^2 - m_{2}^2) \in {\Bbb Z}^*$, and any non zero integer has an absolute value greater or equal to $1$. 

If $\xi \notin {\Bbb Q}$, then the subgroup ${\Bbb Z} + \xi{\Bbb Z}$ is dense in ${\Bbb R}$ and for any $\epsilon > 0$, with $\epsilon < \min(1,\xi)$, there exist two integers $k',j' \in {\Bbb Z}$ such that $0 < k' + j'\xi < \epsilon/8$. One easily sees that necessarily we must have $k'j' < 0$, and thus, without loss of generality, we may assume that we are given two integers $k,j \geq 1$ such that
$$0 < k - j\xi < {\epsilon \over 8}.$$
(This corresponds to the case $k' > 0$ and $j'<0$; when $k' < 0$ and $j' > 0$ one can adapt the argument which follows). Choosing now
$$n := (2k+1,2j-1), \qquad m := (2k-1,2j+1),$$
one verifies that $\mu_{n}(\xi) - \mu_{m}(\xi) = 8k -8j\xi \in (0,\epsilon)$, and thus $\delta(\xi) \leq \epsilon$. We conclude that as a matter of fact we have $\delta(\xi) = 0$. \qed
\medskip

For the case of a rectangle $\Omega = (0,L_{1})\times (0,L_{2})$ we recalled previously that the eigenvalues of the Laplace operator with Dirichlet boundray conditions are given by
$${n_{1}^2 \pi^2\over L_{1}^2} + {n_{2}^2 \pi^2\over L_{2}^2},$$
with $(n_{1},n_{2}) \in {\Bbb N}^*\times{\Bbb N}^*$. Using the above lemma we conclude that when $\Omega$ is such a rectangle and $L_{1}^2/L_{2}^2 \in {\Bbb Q}$, we can take again the exponent $\gamma_{1} = 1$ appearing in \eqref{eq:Cnd-alpha-k}, yielding the same decay estimate for the energy, provided that $\Omega_{1}$ and $\Omega_{2}$ are strips of the form $(\ell_{j},\ell_{j}+\delta_{j})\times(0,L_{2})$ with $0\leq \ell_{j} < \ell_{j} + \delta_{j} \leq L_{j}$ (in which case the exponent $\gamma_{0}$ in \eqref{eq:Cnd-beta-k} is $-1$ when $\epsilon_{2} > 0$, or $0$ when $b_{2} \equiv 0$ and $\epsilon_{1} > 0$).
Thus we can state the following:
\begin{corollary}\label{lem:Decay-N-2-gen}
Assume that $\Omega = (0,L_{1})\times (0,L_{2})$ and $(L_{1}/L_{2})^2 \in {\Bbb Q}$. then the gap between the eigenvalues is bounded below, more precisely, 
$$\lambda_{k+1} - \lambda_{k} \geq {\pi^2 \over L_{1}^2\, q}, \qquad\mbox{if }\, {L_{1}^2 \over L_{2}^2} = {p \over q},$$
with $p$ and $q$ mutually prime. Moreover, if $\Omega_{j} := (\ell_{j},\ell_{j}+\delta_{j})\times(0,L_{2})$ with $0\leq \ell_{j} < \ell_{j} + \delta_{j} \leq L_{1}$, the results of Proposition \ref{lem:Membrane-square} are valid for the solution of \eqref{eq:System-22} on $\Omega$.
\end{corollary}

\begin{remark}
We should point out that when $L_{1}^2/L_{2}^2 \in {\Bbb Q}$, we take the domains $\Omega_{j}$ to be a strip which {\it touches\/} the boundary of $\partial\Omega$, in order to give a lower bound for 
$$\int_{\Omega_{2}}|\nabla\phi(x)|^2\,dx \qquad\mbox{or} \quad \int_{\Omega_{1}}|\phi(x)|^2dx,$$
for all $\phi\in N(L-\lambda_{k}I)$ with $\|\phi\| = 1$. Indeed, if $\Omega_{0}\subset\subset \Omega$ is an open subset, and $\lambda_{k}$ is not a simple eigenvalue, then, as far as we know, it is an open problem to give a lower bound in terms of $\lambda_{k}$ for
$$\int_{\Omega_{0}}|\nabla\phi(x)|^2\,dx \qquad\mbox{or} \quad \int_{\Omega_{0}}|\phi(x)|^2dx,$$
for all $\phi\in N(L-\lambda_{k}I)$ with $\|\phi\| = 1$ (however cf.\ D. S. Grebenkov \& B. T. Nguyen \cite{Grebenkov-Nguyen}, sections 6 and 7).
\qed
\end{remark}
\medskip

Actually one can easily generalize the above Lemma \ref{lem:Gap-1} so that it can be applied to the study of the gap between eigenvalues of the Laplacian on a domain $\Omega \in {\Bbb R}^N$ with $N \geq 3$, which is a product of $N$ intervals. 
The proof of the following statement is straightforward and can be omitted here (with the notations of the corollary, take $(n_{1},n_{2})\neq (m_{1},m_{2})$ and $n_{j} = m_{j}$ for $3 \leq j \leq N$, then apply Lemma \ref{lem:Gap-1}).

\begin{lemma}\label{lem:Gap-2}
Let $N \geq 3$ be an integer and $\xi_{j} > 0$ for $2 \leq j \leq N$. For $n \in ({\Bbb N}^*)^N$ denote $\mu_{n}(\xi) := n_{1}^2 + \sum_{j=2}^N \xi_{j} n_{j}^2$, and 
$$\delta(\xi) := \inf\left\{|\mu_{n} - \mu_{m}| \; ; \; n,m \in ({\Bbb N}^*)^N,\; \mu_{n}\neq\mu_{m}\right\}.$$
Then if there exists $j$ such that $\xi_{j} \notin {\Bbb Q}$ we have $\delta(\xi) = 0$, while if for all $j$ we have $\xi_{j} = p_{j}/q_{j}$ for two mutually prime integers $p_{j},q_{j} \geq 1$, we have
$$\delta(\xi) \geq {1 \over q},$$
where $q$ is the least common multiple of $q_{2},\ldots,q_{N}$. 
\end{lemma}

As a consequence, the result of Corollary \ref{lem:Decay-N-2-gen} is valid in any dimension $N \geq 2$. More precisely, for instance, we can state the following:

\begin{corollary} When 
the Kelvin--Voigt damping region $\Omega_{2}$ is a strip of the form $(\ell_{1},\ell_{1}+\delta_{1})\times(0,L_{2})\times\cdots\times(0,L_{N})$ with $0\leq \ell_{1} < \ell_{1} + \delta_{1} \leq L_{1}$ and $b_{2} \geq \epsilon_{2} >0$ on $\Omega_{2}$, 
 the rate of decay of the energy for the wave equation is also $(1+t)^{-2/5}$, provided that for all $1 \leq i<j \leq N$ the ratios $(L_{i}/L_{j})^2 \in {\Bbb Q}$.
\end{corollary}
\medskip

Next we consider the case of a domain $\Omega := (0,L_{1})\times (0,L_{2})$ such that if $\xi := L_{1}^2/L_{2}^2 \notin {\Bbb Q}$. In this case the exponent $\gamma_{1}$ in \eqref{eq:Cnd-alpha-k} cannot be taken equal to $1$, and a further analysis is necessary. As a matter of fact, Lemma \ref{lem:Gap-1} shows that $\inf_{k \geq 1}(\lambda_{k+1} - \lambda_{k}) = 0$, and therefore the best we can hope for is to find an estimate of the type
$$\lambda_{k+1} - \lambda_{k} \geq c_{0}\, \lambda_{k}^{-\tau} \, ,$$
for some $c_{0} > 0$ and $\tau > 0$ independent of $k$. To this end, we recall that the degree of an algebraic number $\xi$ is the minimal degree of all polynomials $P$ with integer coefficients such that $P(\xi) = 0$.  The following result of 
K. F.~Roth~\cite{Roth-1955} (see Y. Bugeaud \cite{Bugeaud}, chapter 2, Theorem 2.1, page 28) states how well, or rather how badly, as algebraic number of degree greater or equal to two can be approximated by rational numbers:

\begin{theorem}\label{lem:Roth}
{\bf (Roth's Theorem)} Let $\xi > 0$ be an algebraic number of degree greater or equal to two. Then for any $\epsilon > 0$ there exists a positive constant $c(\xi,\epsilon) > 0$ such that for any rational number $p/q$ with $q \geq 1$ one has
\begin{equation}\label{eq:Roth}
\left| \xi - {p \over q}\right| > {c(\xi,\epsilon) \over q^{2+\epsilon}}\, .
\end{equation}
\end{theorem}

We use this result in order to give an estimate of $\lambda_{k+1} - \lambda_{k}$ from below when $\xi = L_{1}^2/L_{2}^2$ is an algebraic number.

\begin{lemma}\label{lem:Estim-Roth}
Assume that $\Omega = (0,L_{1})\times (0,L_{2})$ and that $\xi := L_{1}^2/L_{2}^2$ is an algebraic number of degree greater or equal to two. Then for any $\epsilon > 0$, there exists a constant $c_{0}(\xi,\epsilon)$ such that
\begin{equation}\label{eq:Gap-epsilon}
\lambda_{k+1} - \lambda_{k} \geq \min\left(1, c_{0}(\xi,\epsilon)\, \lambda_{k}^{-1-\epsilon}\right).
\end{equation}
\end{lemma}

\begin{proof}
Let $k \geq 1$ be fixed, and for $m\in  {\Bbb N}^*\times{\Bbb N}^*$ let us denote $\mu_{m} := m_{1}^2 + \xi m_{2}^2$. There exist $m,n\in {\Bbb N}^*\times{\Bbb N}^*$ with $m\neq n$ such that
$$\lambda_{k} = {m_{1}^2\pi^2 \over L_{1}^2} + {m_{2}^2 \pi^2\over L_{2}^2} = {\pi^2 \over L_{1}^2}\,\mu_{m}, \qquad
\lambda_{k+1} = {n_{1}^2\pi^2 \over L_{1}^2} + {n_{2}^2 \pi^2\over L_{2}^2}
= {\pi^2 \over L_{1}^2}\,\mu_{n}.$$
Then we can write
$$\lambda_{k+1} - \lambda_{k} = {\pi^2 \over L_{1}^2}\, (\mu_{n} - \mu_{m}) .$$
If $n_{2} = m_{2}$, we have clearly $\mu_{n} - \mu_{m} = n_{1}^2 - m_{1}^2 \geq 1$, since $n_{1}^2 - m_{1}^2 \in {\Bbb N}^*$. If $m_{2} \neq n_{2}$, using Roth's Theorem \ref{lem:Roth}, by \eqref{eq:Roth} we have, for any $\epsilon >0$,
$$\mu_{n} - \mu_{m} = \left|\left(n_{2}^2 - m_{2}^2\right)\xi + n_{1}^2 - m_{1}^2 \right| \geq {c(\xi,\epsilon) \over |n_{2}^2 - m_{2}^2|^{1+\epsilon}}.$$
It is easily seen that
$$|n_{2}^2 - m_{2}^2| \leq \max(m_{2}^2,n_{2}^2) \leq {1 \over \xi}\max(\mu_{m},\mu_{n}) = {1 \over \xi}\,\mu_{n},$$
and this yields 
$$\mu_{n} - \mu_{m}  \geq {\xi^{1+\epsilon}\, c(\xi,\epsilon) \over \mu_{n}^{1+\epsilon}} = {\xi^{1+\epsilon}\pi^{2(1+\epsilon)} \over L_{1}^{2(1+\epsilon)}}\cdot 
{c(\xi,\epsilon) \over \lambda_{k+1}^{1+\epsilon}}.$$
From this, and the fact that for some constant $c_{*} > 0$ depending only on $L_{1},L_{2}$ we have $\lambda_{k+1} \leq c_{*}\, \lambda_{k}$ (cf Lemma \ref{lem:Lim-1}), one is convinced that \eqref{eq:Gap-epsilon} holds. \qed

\end{proof}

To finish this paper,  we state the following decay estimate when $\xi := L_{1}^2/L_{2}^2 $ is an algebraic number, and we point out that when $\xi$ is a transcendant number we cannot state such a result.

\begin{proposition}\label{lem:Membrane-algeb}
Assume that $N = 2$ and $\Omega:=(0,L_{1})\times(0,L_{2})$ where $\xi := L_{1}^2/L_{2}^2 $ is an algebraic number of degree greater or equal to $2$. Let
$$\Omega_{1} := (a_{1},a_{1}+\delta_{1})\times(a_{2},a_{2} +\delta_{2}), \qquad \Omega_{2} := (b_{1},b_{1}+\delta_{1})\times(b_{2},b_{2} +\delta_{2}),$$
and for $j=1$ or $j=2$, let the  functions $b_{j} \in L^\infty(\Omega)$ be such that $b_{j} \geq \epsilon_{j} \geq 0$ on $\Omega_{j}$, where $\epsilon_{j}$ is a constant. Then, when $\epsilon_{2} > 0$, for any $\epsilon > 0$ there exists a constant $c_{*}(\epsilon) > 0$ such that the energy of the solution of \eqref{eq:System-22} satisfies
\begin{equation}\label{eq:Decay-alg}
\|\nabla u(t,\cdot)\|^2 + \|\partial_{t}u(t,\cdot)\|^2 \leq c_{*}\, 
(1+t)^{-2/(9+\epsilon)}\left[
\|\nabla u_{0}\|^2 + \|u_{1}\|^2\right].
\end{equation}
When $\epsilon_{1} > 0$ and $b_{2} \equiv 0$, one has
\begin{equation}\label{eq:Decay-alg-2}
\|\nabla u(t,\cdot)\|^2 + \|\partial_{t}u(t,\cdot)\|^2 \leq c_{*}\, (1+t)^{-2/(11 + \epsilon)}\left[
\|\nabla u_{0}\|^2 + \|u_{1}\|^2\right].
\end{equation}
\end{proposition}

\begin{proof}
First note that for any $\epsilon > 0$, thanks to Lemmas \ref{lem:Lim-1} and \ref{lem:Estim-Roth}, we have for some constant $c_{0}(\epsilon)$
$$ {\lambda_{k-1} \over \lambda_{k} - \lambda_{k-1}} + {\lambda_{k+1} \over \lambda_{k+1} - \lambda_{k}} \leq c_{0}(\epsilon)\, \lambda_{k}^{2+\epsilon},$$
so that we can take $\gamma_{1} = 2+\epsilon$.

On the other hand, in this case each eigenvalue $\lambda_{k} = 
(n_{1}^2\pi^2/L_{1}^2) + (n_{2}^2\pi^2/L_{2}^2)$ is simple and the corresponding eigenfunction is 
$$\phi_{k}(x) = 2\sin(n_{1}\pi x_{1}/L_{1})\sin(n_{2}\pi x_{2}/L_{2})/\sqrt{L_{1}L_{2}}.$$
Therefore, when $\epsilon_{2} > 0$, it is easily seen that
$$\int_{\Omega_{2}}|\nabla\phi_{k}(x)|^2dx \geq c_{*}\, b_{1}b_{2}\, \pi^2 \left({n_{1}^2 \over L_{1}^2} + {n_{2}^2 \over L_{2}^2}\right) = c_{*}\,b_{1}b_{2}\, \lambda_{k}\, ,$$
so that we can take $\gamma_{0} = -1$, which implies $m = 3 + 2\gamma_{0} + 4\gamma_{1} = 9 + 4\epsilon$, and \eqref{eq:Decay-alg} can be deduced.

When $b_{2}\equiv 0$ and $\epsilon_{1} > 0$, we notice that for some constant $c_{*} > 0$ and all $k \geq 1$ we have
$$\int_{\Omega_{1}}|\phi_{k}(x)|^2dx \geq c_{*},$$
therefore we can take $\gamma_{0} = 0$, hence $m = 3 + 2\gamma_{0} + 4\gamma_{1} = 11 + 4\epsilon$, and \eqref{eq:Decay-alg-2} follows easily.
\qed
\end{proof}

\noindent{\bf Acknowledgment.} This work was essentially done during the visit of the first author at the Mathematics Department of the {\it Beijing Institute of Technology\/} in October 2015. The first author would like to express his thanks to the Mathematics Department of the {\it Beijing Institute of Technology\/} for its hospitality. He would also thank his colleague Vincent S\'echerre (Universit\'e Paris--Saclay, UVSQ, site de Versailles), for helpful discussions.

\end{document}